\documentclass [12pt,letterpaper]{article}
\usepackage{amsmath}    		% Extra math definitions
\usepackage{graphics} 		% PostScript figures
\usepackage{setspace}		% 1.5 spacing
\usepackage{longtable}          % Tables spanning pages
\usepackage{hyperref}
\usepackage{cleveref}
\usepackage{verbatim}
\usepackage{fancyhdr}      % (Ali) To change the heading in the appendix
\usepackage{bm}
\usepackage{color}
\usepackage{pdfpages}
\usepackage{graphicx}
\usepackage{epsfig} % for postscript graphics files
\usepackage{times} % assumes new font selection scheme installed
\usepackage{amsmath} % assumes amsmath package installed
\usepackage{amsthm}
\usepackage{amssymb}  % assumes amsmath package installed
\usepackage{amsfonts}
\usepackage{tikz,tabstackengine,amsmath}
\usetikzlibrary{shapes,shadows,arrows}
\usetikzlibrary{shapes,shapes.geometric,arrows.meta,fit,calc,positioning,automata}
\usepackage{pgfplots}
\usepgfplotslibrary{groupplots}
\usepackage{subfig}
\usepackage{mathtools}
\usepackage{adjustbox}
\usepackage{xcolor}
\usepackage[margin=0.5in]{geometry}
\usepackage [english]{babel}
\usepackage [autostyle, english = american]{csquotes}
\MakeOuterQuote{"}
\advance\oddsidemargin0.6in

\usepackage{bm}

\usepackage[utf8]{inputenc}

\usepackage{enumitem}

\usepackage[noadjust]{cite}

\newtheorem{theorem}{Theorem}

\newtheorem{proposition}[theorem]{Proposition}

\theoremstyle{definition}

\theoremstyle{remark}
\newtheorem{remark}{Observation}

\theoremstyle{assumption}
\newtheorem{assumption}{Assumption}

\newcommand{\mb}{\mathbb}
\newcommand{\mc}{\mathcal}
\newcommand{\br}{\breve}
\newcommand{\tf}{\textsf}

\newcommand{\tbf}{\textbf}
\newcommand{\I}{\mathcal{I}}

\begin{document}
\title{$\epsilon$-Nash Equilibria for Major Minor LQG Mean Field Games with Partial Observations of All Agents}
\author{Dena~Firoozi% <-this % stops a space
\footnote{D. Firoozi is with the Department of Decision Sciences, HEC Montr\'eal, Montreal, QC, Canada, e-mail: dena.firoozi@hec.ca.}~~and~Peter~E.~Caines % <-this % stops a space
\footnote{P. E. Caines is with the Centre for Intelligent Machines (CIM) and the Department of Electrical and Computer Engineering (ECE), McGill University, Montreal, QC, Canada, e-mail: peterc@cim.mcgill.ca.}}% <-this % stops a space
%%\thanks{J. Doe and J. Doe are with Anonymous University.}% <-this % stops a space
%%\thanks{Manuscript received April 19, 2005; revised August 26, 2015.}
\date{}
\maketitle
\begin{abstract} The partially observed major minor LQG and nonlinear mean field game (PO MM LQG MFG) systems where it is assumed the major agent's state is partially observed by each minor agent, and the major agent completely observes its own state have been analysed in the literature. In this paper, PO MM LQG MFG problems with general information patterns are studied where (i) the major agent has partial observations of its own state, and (ii) each minor agent has partial observations of its own state and the major agent's state. The assumption of partial observations by all agents leads to a new situation involving the recursive estimation by each minor agent of the major agent's estimate of its own state. For a general case of LQG MFG systems, the existence of $\epsilon$-Nash equilibria together with the individual agents' control laws yielding the equilibria are established via the Separation Principle.
\end{abstract}
%\begin{IEEEkeywords}
%estimates of estimates, partial observations, LQG MFG
%\end{IEEEkeywords}

% \fi
%\IEEEpeerreviewmaketitle

\section{Introduction}
Mean field game theory (MFG) studies the existence of approximate Nash equilibria and the corresponding individual strategies for stochastic dynamical systems in games involving a large number of agents. Basically, the theory exploits the relationship between the large finite and the corresponding infinite limit population problems. The equilibria are termed $\epsilon$-Nash equilibria and are generated by the local, limited information control actions of each agent in the population. The control actions constitute the best response of each agent with respect to the behaviour of the mass of agents. Moreover, the approximation error, induced by using the MFG solution, converges to zero as the population size tends to infinity.

The analysis of this set of problems originated in \cite{HuangCDC2003a, HuangCDC2003b, HuangCIS2006, HuangTAC2007}, and independently in \cite{Lasry2006a, Lasry2006b, Lasry2007}. Many extensions and generalizations of MFGs exist, principally the probabilistic formulation \cite{CarmonaDelarueBook2018}, the master equation approach \cite{CardaliaguetMasterEqBook2019} and mean-field-type control theory \cite{BensoussanBook2013}. In \cite{Huang2010, Huang2012} the authors analyse and solve the completely observed (CO) linear quadratic Gaussian (LQG) systems case where there is a major agent (i.e. non-asymptotically vanishing as the population size goes to infinity)  together with a population of minor  agents (i.e. individually asymptotically negligible). The new feature in this case is that the mean field becomes stochastic but by minor agent state extension the existence of  closed-loop $\epsilon$-Nash equilibria is established together with the individual agents' control laws that yield the equilibria \cite{Huang2012}. A convex analysis method is utilized in \cite{FCJ-Convex2018} to retrieve the solutions of \cite{Huang2010}, where no assumption is imposed on the evolution of the mean field a priori. For such systems, among others, \cite{HuangMA2020} presents a multi-scale analysis and the notion of asymptotic solvability, \cite{FirooziPakniyatCainesCDC2017, FPC-arXiv2018} present a hybrid optimal control approach to address switching or cessation of agents, and \cite{Kordonis2015} considers the case with a random number of minor agents. The CO MM nonlinear (NL) MFG problem is treated in \cite{NourianSiam2013}. Using the probabilistic approach to MFGs, \cite{CarmonaZhu2016,CarmonaWang2017} establish the existence of open-loop and closed-loop $\epsilon$-Nash equilibria for a general MM MFG and provide explicit solutions for an LQG case.  The works \cite{LasryLions2018, CardaliaguetCirant2018} characterize the Nash equilibrium for a general MFG system with one major agent and an infinite number of minor agents via the MFG  Master Equations. It is to be noted that for the LQG case it has been demonstrated in \cite{HuangCIS2020} that the LQG MM MFG Master Equations yield the original LQG MM MFG equations of  \cite{Huang2010}. (Another line of research characterizes a Stackelberg equilibrium between the major agent and the minor agents, see e.g. \cite{BensoussanSICON2017,BasarMoon2018}.)

In the purely minor agent case the mean field is deterministic and this obviates the need for observations on other agents' states. This is a separate issue from that of an agent estimating its own state from partial observations on that state, see \cite{HuangMTNS2006}. However, when a system has a major agent whose state is partially observed the standard MFG procedure for generating a Nash equilibrium needs to be extended for each minor agent by including an estimate of the major agent's state generated by that agent. In \cite{Kizilkale2014, KizilkaleTAC2016}, partially observed LQG mean field games with major and minor agents (PO MM LQG MFG) have been investigated and in \cite{SenCDC2014, SenSIAM2016, CainesSen2019}, a nonlinear generalization of this problem is considered. The main results in those papers are obtained with the assumptions that (i) the major agent's state is partially observed by the minor agents and (ii) the major agent has complete observations on its own state. 

An initial investigation of the case where assumption (i) holds but the major agent has partial observations on its own state was presented in \cite{FirooziCDC2015}. A thorough investigation of this case is given in the present paper, while \cite{FJCCDC2018, FJC-arXiv2018} analyse the case where all agents partially observe a common process, and \cite{FirooziCainesCDC2019} studies the case with two major agents  and identifies the partial information patterns which lead to tractable solutions. The main contributions of the current paper are summarized as follows: 
\begin{itemize}
\item PO MM LQG MFG problems with general information patterns are studied where (i) the major agent has partial observations of its own state, and (ii) each minor agent has partial observations of its own state and the major agent's state. 

 \item In the PO MM LQG MFG theory presented here the major agent recursively estimates its own state, and each minor agent recursively estimates its own state and the major agent's estimate of its own state (in order to estimate the major agent's feedback control input). In addition, both the major agent and minor agents generate estimates of the system's mean field.
 
 We remark that an infinite regress does not happen here due to the asymmetric major minor (MM) feature of the MFG problem.  
%\item MFG theory is extended to cover the general case of indefinite LQG MFG systems which alleviates positive definiteness condition of weight matrices in linear quadratic cost functionals. 

\item The existence of $\epsilon$-Nash equilibria together with the individual agents' control laws yielding the equilibria is established; this is achieved in the PO MM LQG case by an application of the Separation Principle which also yields computationally tractable solutions while in the nonlinear case is far more complex (see \cite{SenCDC2014, SenSIAM2016}). 

\item This extension of the situation in \cite{KizilkaleTAC2016}, where only assumption (ii) holds, is in particular motivated by optimal execution problems in financial markets where there exist one institutional trader (interpreted as major agent) and a large population of high frequency traders (interpreted as minor agents) who attempt to maximize their own wealth. To obtain the Nash equilibrium best response trading strategy, each minor agent estimates the major agent's inventory and trading rate based on its partial observations of market state and this entails the estimation of the major agent's self estimates. The reader is referred to the works \cite{FirooziCDC2016,FirooziISDG2017}  for more details on financial applications.
\end{itemize}
The rest of the paper is organized as follows. Section \ref{sec:PO MM MFG LQG} introduces partially observed major-minor (PO MM) LQG MFG systems. The estimation and control problems for PO MM LQG MFG systems are addressed in Section \ref{sec:Est Ctrl PO MM MFG LQG}. The simulation results and the concluding remarks are presented in Section \ref{sec:simulation} and Section \ref{sec:conclusions}, respectively.
\section{Partially Observed Major-Minor LQG MFG Systems} \label{sec:PO MM MFG LQG}
 A class of major-minor LQG MFG (MM LQG MFG) systems including a large population of $N$ stochastic dynamic minor agents with a stochastic dynamic major agent is considered.  \subsection{Dynamics}
 The dynamics of the major and minor agents in the class of systems under consideration are, respectively, given by
\begin{align}
dx_0 &= [A_0 x_0  +  B_0 u_0] dt + D_0 d w_0,   \label{MajorAgentEq}\\
dx_i  &= [A(\theta_i) x_i + B(\theta_i) u_i + G x_0] dt + D dw_i,  \label{MinorAgentEq}
\end{align}
where $t \geq 0$, $ 1 \leq i \leq N < \infty$, $\theta_i \in \Theta \subset \mb{R}^{n\times(n+m)}$, where $\Theta$ is a parameter set and $\theta_i$ determines the pair $(A(\theta_i),B(\theta_i))$ of $i$-th minor agent ($\mc{A}_i$). Here $x_i \in \mathbb{R}^n,~ 0 \leq i \leq N$, are the states, $u_i \in \mathbb{R}^m,~ 0 \leq i \leq N$, are the control inputs, $w = \lbrace w_i,~ 0 \leq i \leq N \rbrace$ denotes the set of $(N+1)$ independent standard Wiener processes in $\mathbb{R}^r$ on an underlying probability space $(\Omega, \mathcal{F}, P)$ which is sufficiently large that $w$ is progressively measurable with respect to the filtration $\mathcal{F}^w := (\mathcal{F}_t^{w};~t \geq 0) \subset \mathcal{F}$, and $\mathbb{E} w_i(t) w_i(t)^T = I_r\, t,$ with $I_r$ denoting the identity matrix of size $r$.
\begin{assumption} \label{IntialStateAss}
The initial states $\lbrace x_i(0),~ 0 \leq i \leq N\rbrace$ defined on $(\Omega, \mathcal{F}, P)$ are identically distributed, mutually independent and also independent of $\mathcal{F}_{\infty}^{w}$, with $\mathbb{E}x_i(0)=0$. Moreover, $\sup_{i} \mathbb{E}\Vert x_i(0)\Vert^2 \leq c < \infty $, $0 \leq i \leq N<\infty$, with $c$ independent of $N$.
\end{assumption} 
The matrices $A_0, B_0, D_0, G,$ and $D$ are constant matrices of appropriate dimensions. From \eqref{MinorAgentEq}, $A(.)$ and $B(.)$ depend on the parameter $\theta$ which specifies the minor agent's type. Minor agents are given in $K$ distinct types with $1 \leq K < \infty$. The notation $\mathcal{I}_k$ is defined as 
\begin{equation}
\mathcal{I}_k = \lbrace i : \theta_i = \theta^{(k)},\, 1 \leq i \leq N \rbrace , \quad 1 \leq k \leq K\nonumber,
\end{equation}
where $\theta^{(k)}\in\Theta$ and the cardinality of $\mathcal{I}_k$ is denoted by $N_k = |\mathcal{I}_k|$. Then, $\pi^{N} = (\pi_{1}^{N},...,\pi_{K}^N),~ \pi_k^N = \tfrac{N_k}{N} \in \mathbb{R}, ~ 1 \leq k \leq K$, denotes the empirical distribution of the parameters $(\theta_1,...,\theta_N)$ sampled independently of the initial conditions and Wiener processes of the agents $\mc{A}_i, 1 \leq i \leq N$. \begin{assumption} \label{EmpiricalDistAss}
There exists $\pi$ such that $\mbox{lim}_{N \rightarrow \infty} \pi^N = \pi $ a.s. 
\end{assumption}

\subsection{Cost Functionals}
The individual (finite) large population infinite horizon cost functional for the major agent $\mc{A}_0$ is specified by
\begin{align} \label{MajorCostLrgPop}
J_{0}^{N}(u_0, u_{ -0}) = \mathbb{E} \int_{0}^{\infty}& e^{- \rho t} \Big \lbrace {\Vert x_0 - \Phi(x^{(N)})\Vert ^2_{Q_0} + \Vert u_0 \Vert_{R_0}^2} \Big \rbrace dt, \\
\Phi(x^{(N)}) ~&:=~ H_0 x^{(N)} + \eta_0, \nonumber
\end{align}
where $R_0 > 0$, and the individual (finite) large population infinite horizon cost functional for a minor agent $\mathcal{A}_i, 1 \leq i \leq N$, is given by
\begin{align} \label{MinorCostLrgPop}
J_{i}^{N}(u_i, u_{-i}) = &\mathbb{E} \int_{0}^{\infty} e^{-\rho t} \Big \lbrace \Vert x_i - \Psi(x^{(N)},x_0) \Vert_{Q}^{2} + \Vert u_i \Vert_R^2 \Big\rbrace dt, \\
&\Psi(x^{(N)},x_0) ~:=~ H_1 x_0 + H_2 x^{(N)} + \eta, \nonumber
\end{align}
where $R >0$. We note that the major agent $\mathcal{A}_0$ and minor agents $\mathcal{A}_i,\, 1 \leq i \leq N$, are coupled with each other through the average term $x^{(N)} = \frac{1}{N} \sum_{i=1}^{N} x_i$ in their cost functionals given by \eqref{MajorCostLrgPop}-\eqref{MinorCostLrgPop}.  
\subsection{Observation Processes}
The major agent's partial observations $y_0 \in \mathbb{R}^{p}$ is given by
\begin{align} \label{MajorPartialObs}
dy_0  = L_0 [x_0^T,\, (x^{(N)})^T]^T dt + \sigma_{v_0}dv_0, 
\end{align}
where $v_0$ is a standard Wiener process in $\mb{R}^\ell$ with $\mathbb{E} [v_0(t) v_0(t)^T] = I_{\ell}t$, and matrix $L_0$ is given by 
\begin{align}
L_0 = \left[ \begin{array}{cc}
{l}_0^1 & 0_{p \times n}\\
\end{array} \right],
\end{align}
with ${l}_0^1, \sigma_{v_0}$ being constant matrices of appropriate dimension.
The partial observation process $y_i \in \mathbb{R}^p$ for a minor agent $\mathcal{A}_i,~1\leq i \leq N$, of type $k,\, 1 \leq k \leq K$, is given by
 \begin{align} \label{MinorPartialObs}
{dy}_i = L_k [x^T_i, \,  x^T_0, \, (x^{(N)})^T]^T dt + \sigma_{v}dv_i,
\end{align}
where $\{v_i,\, 1\leq i \leq N\}$ denotes the set of $N$ independent standard Wiener processes in $\mb{R}^{\ell}$ with $\mathbb{E}[v_i(t) v_i(t)^T]=I_{\ell}t$, and matrix $L_k$ is given by  
\begin{align}
L_k = \left[ \begin{array}{ccc}
l^1_k & l^2_k & 0_{p \times n} 
 \end{array} \right],
\end{align}
where $l^1_k, \, l^2_k$, and $\sigma_v$ are constant matrices of appropriate dimension. 

\begin{assumption}\label{independence_ass}$\{v_i, \,0 \leq i \leq N \}$ are (N+1) independent Wiener processes that are independent of the Wiener processes $\{w_i, \, 0 \leq i\leq N \}$ and the initial states $\{x_i(0), \,0 \leq i \leq N\}$. 
 \end{assumption}

\begin{assumption}[Major Agent $\sigma$-Fields and Linear Controls] \label{ass:MajorControl}
 The family of partial observation information sets $\mathcal{F}^y_0$ is defined to be the increasing family of $\sigma$-fields of partial observations $\lbrace \mathcal{ F}^{y}_{0,t}, t \geq 0 \rbrace$ generated by the major agent $\mathcal{A}_0$'s partial observations $(y_0(\tau), 0 \leq \tau \leq t)$ on its own state as given in (\ref{MajorPartialObs}). For the major agent $\mathcal{A}_0$ the set of control inputs $\mathcal{U}_{0,y}^L$ is defined to be the collection of linear feedback control laws adapted to $\lbrace \mathcal{F}^y_{0,t}, t \geq 0 \rbrace$.
\end{assumption}

\begin{assumption}[Minor Agent $\sigma$-Fields and Linear Controls]\label{ass: MinorContrAction} The family of partial observation information sets $\mathcal{F}_i^y,1 \leq i \leq N$, is defined to be the increasing $\sigma$-fields $\lbrace \mathcal{F}^y_{i,t}, t \geq 0 \rbrace$ generated by the minor agent $\mathcal{A}_i$'s partial observations $(y_i(\tau), 0 \leq \tau \leq t)$, on its own state and the major agent's state, as given in (\ref{MinorPartialObs}). For each minor agent $\mathcal{A}_i, 1 \leq i \leq N$, the set of control inputs ${\mathcal{U}}^{L}_{i,y}$ is defined to be the collection of time-invariant linear feedback control laws adapted to $\lbrace \mathcal{F}^y_{i,t}, t \geq 0 \rbrace$.
\end{assumption}
The set of control inputs $\mc{U}^{N,L}_y$ is defined to be the collection of linear feedback control laws adapted to $\mc{F}^{N,y}_t = \{ \bigvee_{i=0}^{N}\mc{F}^y_i  \}$. 

We note that for simplicity of notation throughout the paper, time arguments for deterministic and stochastic processes may be dropped, as in \eqref{MajorAgentEq}-\eqref{MinorPartialObs}. Further, the analysis in the next section could be directly applied to the case where matrices $G$, $D$, $Q$, $H_1$, $H_2$, $\eta$, $R$, and $\sigma_v$ depend on the type $k,\, 1\leq k \leq K,$ of minor agents. 

\section{Estimation and Control Solutions for PO MM LQG MFG Systems} \label{sec:Est Ctrl PO MM MFG LQG}
 In this section we present the solution to partially observed (PO) MM LQG MFG problems where it is assumed that the major agent partially observes its own state, and each generic minor agent partially observes its own state and the major agent's state. The problem is first solved in the infinite population case which is far simpler to solve than the finite large population problem. Because the agents in the infinite population case are decoupled and therefore the problem reduces to the LQG tracking problem whose solution is given in \textit{Theorem \ref{thm:StochIndefiniteLQ}}. Subsequently, the $\epsilon$-Nash equilibrium property is established in \textit{Theorem \ref{Thm: POLQGMM-MFG}} for the system when the infinite population control laws are applied to the finite large population PO MM LQG MFG system. 
  
 The following theorem is a restriction to the constant matrix parameter case of the general result in \cite{XYZBook1999}.  
% \hfill $\square$\\
 %After applying the mean field methodology to decouple the agents, the problem of obtaining the best response trading strategy is transformed to a stochastic indefinite LQ problem that is solved for using the following theorem which is a restriction to the constant matrix parameter case of the general result in \cite{XYZBook1999}.  
\begin{theorem}[Stochastic LQ Problem \cite{XYZBook1999}]\label{thm:StochIndefiniteLQ}
Let $\br T>0$ be given. For any $(\br s,\br y)\in [0,\br T) \times \mathbb{R}^n$, consider the following linear system
 \begin{equation} \label{generalDynamics}
  d\br x=\big [\br A \br x + \br B \br u + \br b \big]dt +\big[\br C \br x+ \br D \br u+ \br \sigma \big]d\br w, 
 \end{equation}
where $t \in [\br s,\br T], ~\br x(\br s)=\br y$ and $\br A$, $\br B$, $\br C$, $\br D$, $\br b$, $\br \sigma$ are matrix valued functions of suitable sizes, $\br w(.) \in \mb{R}^r$ is a standard Wiener process. Moreover, $\mc{F}_t = \sigma \{\br w(\tau), 0\leq \tau \leq t \}$, and $\br u(.) \in \mc{U}$, where $\mc{U}$ is the set of all $\mc{F}_t$-adapted $\mb{R}^m$-valued processes such that $\mb{E}\int_{0}^{T}\hspace{0mm}\Vert u(t) \Vert^2 dt\hspace{0mm}< \infty$. 

A quadratic cost functional is given by
\begin{multline}\label{generalCostFunc}
J(\br s, \br y,\br u(.)) = \mathbb{E} \Big \{ \frac{1}{2} \int_{0}^{\br T} \big[ \langle \br P \br x(t), \br x(t) \rangle + \langle \br N \br x(t), \br u(t) \rangle \\+ \langle \br R \br u(t), \br u(t) \rangle \big] dt + \frac{1}{2} \langle \br {\bar{P}} \br x(\br T), \br x(\br T) \rangle \Big \},  
\end{multline}
with $\br{\bar{P}} \geq 0$, $\br P$, $\br N$ and $\br R$ being $\mc{S}^{n}$, $\mc{S}^{n}$, $\mathbb{R}^{m\times n}$ and $\mc{S}^{m}$-valued functions of time, respectively, and, where $\mathcal{S}^n$ denotes symmetric matrix space of size $n$. Moreover, $P-N^TR^{-1}N \geq 0$. 
\begin{comment}
%\begin{assumption}
\begin{align*}
&A, C \in L^{\infty}(0,T; \mathbb{R}^{n\times n}), B, D, N^T \in L^{\infty}(0,T; \mathbb{R}^{n\times m}), \nonumber \\ 
&\sigma \in L^{\infty}(0,T;\mathbb{R}^{n \times r} ),~  b \in L^2(0, T; \mathbb{R}^n),  
\nonumber \\
&P \in L^{\infty}(0, T; \mathcal{S}^n), R \in L^{\infty}(0,T; \mathcal{S}^m), \bar{P}\in \mathcal{S}^n, \nonumber 
\end{align*}
%\end{assumption}
\end{comment}

We also denote the set of all $\mb{R}^n$-valued continuous functions defined on $[s,T]$ by $\mathbf{C}([s,T];\mb{R}^n)$. Then, let $\br{\Pi}(.)\in \mathbf{C}([\br s,\br T]; \mathcal{S}^n)$ be the solution of the Riccati equation
\begin{multline} \label{generalRiccatiEq}
\dot{\br{\Pi}} + \br{\Pi} \br{A} + \br{A}^{T} \br{\Pi} + \br{C}^{ T} \br P \br C + \br P - (\br B^{ T} \br{\Pi} + \br{N} + \br{D}^{T}\br{\Pi} \br{C})^{T} (\br{R} +\br{D}^{T} \br{\Pi} \br{D})^{-1}\\ \times (\br{B}^{T}\br{\Pi} + \br{N} + \br{D}^{T}\br{\Pi} \br{C}) =0, \quad a.e. t \in[\br s,t],\,\,
\br{\Pi}(\br T)= \br{\bar{P}},
\end{multline}
where $\br R+\br D^T\br \Pi \br D > 0,~ a.e.\, t \in [\br s,\br T]$,
and $\br s(.) \in C([\br s,\br T]; \mathbb{R}^n)$ be the solution of the offset equation given by
\begin{multline*}\label{generalOffsetEq}
\hspace{-3mm}\dot{\br s} + [\br A-\br B(\br R+\br D^{T}\br{\Pi} \br D)^{-1} (\br B^{T}\br P+\br s+\br D^T\br P\br C)]^{T}\br s 
+[ \br C-\br D(\br R+\br D^{T}\br{\Pi} \br{D})^{-1}\\\hspace{-1mm}(\br{B}^T\br{\Pi} + \br{N} + \br{D}^T \br{\Pi} \br C)]^T \br{\Pi} \br{\sigma} 
 +\br{\Pi} \br{b} =0, \quad a.e. \, t \in [\br s,\br T],\,\, \br s(\br T) = 0.
\end{multline*}
Let us define $\br{\Psi} := (\br R+\br D^T \br{\Pi} \br{D})^{-1} [\br{B}^T \br{\Pi} + \br{N} + \br{D}^T \br{\Pi} \br{C}]$,
and $\br{\psi} := (\br R + \br D^T \br{\Pi} \br{D})^{-1} [\br B^T \br s + \br D^T \br{\Pi} \br{\sigma}]$.
Then the stochastic LQ problem \eqref{generalDynamics}-\eqref{generalCostFunc} is solvable at $\br{s}$ with the optimal control $\br{u}^{\circ}(.)$ being in the state feedback form as in 
\begin{equation*}
\br{u}^{\circ}(t) = - \br{\Psi}(t)\br{x}(t) - \br{\psi}(t),\quad t \in [\br s,\br T].  
\end{equation*}
\hfill $\square$
\end{theorem}
Henceforth we discuss the stochastic optimal control problem for the major agent, and a generic minor agent. 
\subsection{Mean Field Evolution}
We introduce the empirical state average as
\begin{equation}
x^{(N_k)} = \frac{1}{N_k} \sum_{j\in \I_k} x_{j}^{k},~~~~~1 \leq k \leq K, \nonumber 
\end{equation}  
and write $(x^{(N)})^T= [(x^{(N_1)})^T, ..., (x^{(N_K)})^T]$, where the point-wise in time quadratic mean limit of $x^{(N)}$ as $N\rightarrow \infty$, when it exists, is called the mean field of the system and is denoted by $\bar{x}^T= [(\bar{x}^1)^T, ..., (\bar{x}^K)^T]$. We consider for each minor agent $\mathcal{A}_i$ of type $k$, $1 \leq k \leq K$, a uniform (with respect to $i$ in any subpopulation $k,\, 1\leq k \leq K$) feedback control $u_i^k \in \mathcal{U}_{i,y}^{L} $, which is a function of \\(i) bounded time-invariant matrix $m_k \in \mb{R}^m$,\\ and the minor agent's estimate of:
\begin{itemize}
\item[(ii)]  its own state, i.e. $\hat{x}_{i|\mathcal{F}^y_i}  := \mathbb{E}_{|\mathcal{F}^y_i} x_i = \mathbb{E}\{  x_i | \mathcal{F}^y_i \}$,
\item[(iii)] the major agent's state, i.e. $\hat{x}_{0|\mathcal{F}^y_i}  := \mathbb{E}_{|\mathcal{F}^y_i} x_0 = \mathbb{E}\{  x_0 | \mathcal{F}^y_i  \}$,
\item[(iv)] $x_j,\, 1 \leq j \leq N,\, j\neq i$, i.e. $\hat{x}_{j|\mathcal{F}^y_i}  := \mathbb{E}_{|\mathcal{F}^y_i} x_j = \mathbb{E}\{ x_j | \mathcal{F}^y_i  \}$,
\item[(v)] the major agent's estimate of its own state, i.e. $(\hat{x}_{0|\mathcal{F}^y_0})_{|\mathcal{F}^y_i} := \mathbb{E}_{| \mathcal{F}^y_i} \hat{x}_{0|\mathcal{F}^y_0} =\mathbb{E} \{ \hat{x}_{0|\mathcal{F}^y_0} |  \mathcal{F}^y_i  \},$
\item[(vi)] the major agent's estimate of $x_j,\, 1\leq j \leq N$, i.e. $(\hat{x}_{j|\mathcal{F}^y_0})_{|\mathcal{F}^y_i} := \mathbb{E}_{|\mathcal{F}^y_i} \hat{x}_{j|\mathcal{F}^y_0} =\mathbb{E} \{\hat{x}_{j|\mathcal{F}^y_0} |  \mathcal{F}^y_i  \}.$
\end{itemize}
Hence $u_i^k$ is given by
\begin{multline}\label{generalMinorCntrl}
u_i^k =  L_1^k \hat{x}_{i|\mathcal{F}^y_i}^k + L_2^k \hat{x}_{0|\mathcal{F}^y_i} + \sum_{l=1}^{K} \sum_{j\in \I_l} L_3^{k,l} \hat{x}_{j|\mathcal{F}^y_i}^l \allowdisplaybreaks\\+ L_4^k (\hat{x}_{0|\mathcal{F}^y_0})_{|\mathcal{F}^y_i} + \sum_{l=1}^{K} \sum_{j\in \I_l} L_5^{k,l} (\hat{x}_{j|\mathcal{F}^y_0}^l)_{|\mathcal{F}^y_i} + m_k,
\end{multline}
for bounded matrices $L_1^k, \,L_2^k,\, L_3^{k,l}$, and $L_4^k$ of appropriate dimension, and where bounded matrices $L_3^{k,l}$, $L_5^{k,l}$ satisfy $N_l L_3^{k,l} \rightarrow \bar{L}^{k,l}_3$, $N_l L_5^{k,l} \rightarrow \bar{L}^{k,l}_5$ as $N_l \rightarrow \infty$ for all $k, \, 1\leq k \leq K$. All the coefficient matrices in \eqref{generalMinorCntrl} are time-invariant as per %\textit{Assumption \ref{ass: MinorContrAction}}.% (concerning  the validity of this assumption see \cite{GERADreport2020}).
the following argument. In the complete observations case each agent $\mathcal{A}_i$'s extended state  $(x_i,  {x_j}, 1\leq j \leq N, j \neq i, x_0)$ is generated by the extended dynamics with constant coefficients, as defined by \eqref{MajorAgentEq} and \eqref{MinorAgentEq}, where it shall be assumed that all other agents are using the same feedback strategy. Consequently, the optimal feedback gain for $\mathcal{A}_i$ for the discounted constant coefficient cost function \eqref{MajorCostLrgPop}, depends upon the steady state solution of control Riccati and offset equations with time invariant coefficients and hence is itself time invariant \cite{OptCntrl_Anderson_Moore}. In the partial observations case, the minor agent $\mathcal{A}_i$ generates estimates of its extended state via a linear filter whose drift coefficients (but not, in general, its diffusion coefficients) are the time invariant coefficients of the original system. Hence the feedback gain coefficients in the partially observed case in \eqref{generalMinorCntrl} are time invariant (see \Cref{appendix} and \cite{CainesBook1988}).

 Substituting \eqref{generalMinorCntrl} in \eqref{MinorAgentEq} for $1\leq i \leq N$ and $1 \leq k \leq K$  yields
\begin{multline}
dx_i^k = \Big[A_k x_i^k + G x_0 + B_k \Big( L_1^k \hat{x}_{i|\mathcal{F}^y_i}^k + L_2^k \hat{x}_{0|\mathcal{F}^y_i} + \sum_{l=1}^{K} N_l L_3^{k,l} \hat{x}_{|\mathcal{F}^y_i}^{(N_l)} \allowdisplaybreaks\\+  L_4^k (\hat{x}_{0|\mathcal{F}^y_0})_{|\mathcal{F}^y_i} + \sum_{l=1}^{K} N_l L_5^{k,l} (\hat{x}_{|\mathcal{F}^y_0}^{(N_l)})_{|\mathcal{F}^y_i} + m_k\Big) \Big] dt + D dw_i. 
\end{multline}
Then we take the average over the subpopulation $k$ to obtain 
\begin{multline}\label{HeuristicAve}
dx^{(N_k)} = \Big[A_k x^{(N_k)} + G x_0+ B_k \Big(L_1^k \frac{1}{N_k}\sum_{i\in \I_k}\hat{x}_{i|\mathcal{F}^y_i}^k +  L_2^k\frac{1}{N_k}\sum_{i\in \I_k}\hat{x}_{0|\mathcal{F}^y_i} \allowdisplaybreaks\\+ \sum_{l=1}^{K} N_l L_3^{k,l} \frac{1}{N_k}\sum_{i\in \I_k}\hat{x}_{|\mathcal{F}^y_i}^{(N_l)}+  L_4^k \frac{1}{N_k}\sum_{i\in \I_k}(\hat{x}_{0|\mathcal{F}^y_0})_{|\mathcal{F}^y_i} \allowdisplaybreaks\\ +  \sum_{l=1}^{K} N_l L_5^{k,l} \frac{1}{N_k}\sum_{i\in \I_k}(\hat{x}_{|\mathcal{F}^y_0}^{(N_l)})_{|\mathcal{F}^y_i} +  m_k \Big)\Big] dt + D\frac{1}{N_k}\sum_{i\in \I_k}dw_i.
\end{multline}
To compute the average of the estimation terms in \eqref{HeuristicAve}, we use the state decomposition 
\begin{align}
\begin{bmatrix}
\hat{x}_{i|\mathcal{F}^y_i}\\
\hat{x}_{0|\mathcal{F}^y_i}\\
\hat{x}_{|\mathcal{F}^y_i}^{(N_l)}\\
(\hat{x}_{0|\mathcal{F}^y_0})_{|\mathcal{F}^y_i}\\
(\hat{x}_{|\mathcal{F}^y_0}^{(N_l)})_{|\mathcal{F}^y_i}
 \end{bmatrix}  = \begin{bmatrix}
 \hat{x}_{i|\mathcal{F}^y_i} -x_i\\
\hat{x}_{0|\mathcal{F}^y_i} - x_0\\
\hat{x}_{|\mathcal{F}^y_i}^{(N_l)} - x^{(N_l)}\\
( \hat{x}_{0|\mathcal{F}^y_0})_{|\mathcal{F}^y_i} - \hat{x}_{0|\mathcal{F}^y_0}\\
(\hat{x}_{|\mathcal{F}^y_0}^{(N_l)})_{|\mathcal{F}^y_i} - \hat{x}_{|\mathcal{F}^y_0}^{(N_l)}
 \end{bmatrix} +  \begin{bmatrix}
 x_i\\
 x_0\\
 x^{(N_l)}\\
 \hat{x}_{0|\mathcal{F}^y_0}\\
  \hat{x}_{|\mathcal{F}^y_0}^{(N_l)}
 \end{bmatrix},
\end{align}
which we denote equivalently in the compact form as in
\begin{equation}\label{EstDecompCompact}
\hat{x}_{i|\mc{F}^y_i}^{ex} = -\tilde{x}^{ex}_{i} + x_i^{ex},
\end{equation} 
for $1\leq i \leq N$. Accordingly, we rewrite \eqref{HeuristicAve} for $1 \leq k \leq K$ as 
\begin{multline}\label{HeuristicAveEstErr}
dx^{(N_k)} = \Big[A_k x^{(N_k)} + G x_0 + B_k\Big(L_1^k \frac{1}{N_k}\sum_{i\in \I_k}{x}_{i}^k + L_2^k{x}_{0} + \sum_{l=1}^{K} N_l L_3^{k,l} {x}^{(N_l)}\allowdisplaybreaks\\+ L_4^k \hat{x}_{0|\mathcal{F}^y_0} +  \sum_{l=1}^{K} N_l L_5^{k,l} \hat{x}_{|\mathcal{F}^y_0}^{(N_l)}+  m_k \Big)\Big] dt \allowdisplaybreaks\\ -B_k \Big[L_1^k \frac{1}{N_k}\sum_{i\in \I_k}(x_i^k-\hat{x}_{i|\mathcal{F}^y_i}^k) +  L_2^k\frac{1}{N_k}\sum_{i\in \I_k}(x_0-\hat{x}_{0|\mathcal{F}^y_i}) \allowdisplaybreaks\\+ \sum_{l=1}^{K} N_l L_3^{k,l} \frac{1}{N_k}\sum_{i\in \I_k}\big(x^{(N_l)}-\hat{x}_{|\mathcal{F}^y_i}^{(N_l)}\big) +  L_4^k \frac{1}{N_k}\sum_{i\in \I_k}\big(\hat{x}_{0|\mathcal{F}^y_0}-(\hat{x}_{0|\mathcal{F}^y_0})_{|\mathcal{F}^y_i}\big)\\+  \sum_{l=1}^{K} N_l L_5^{k,l} \frac{1}{N_k}\sum_{i\in \I_k}\big(\hat{x}_{|\mathcal{F}^y_0}^{(N_l)}-( \hat{x}_{|\mathcal{F}^y_0}^{(N_l)})_{|\mathcal{F}^y_i}\big)\Big]dt + D\frac{1}{N_k}\sum_{i\in \I_k}dw_i.
\end{multline}

From \eqref{HeuristicAveEstErr} as $N \rightarrow \infty$ we obtain the convergence in quadratic mean (q.m.) to the solution to 
\begin{multline}\label{HeuristicAveEstErrLimit}
d\bar{x}^k  = \Big[(A_k + B_k L_1^k) \bar{x}^k + (G+B_k L_2^k) x_0+ B_k \Big(\sum_{l=1}^{K} \bar{L}_3^{k,l} \bar{x}^l +  L_4^k \hat{x}_{0|\mathcal{F}^y_0}+  \sum_{l=1}^{K} \bar{L}_5^{k,l} \hat{\bar{x}}_{|\mathcal{F}^y_0}^l +  m_k \Big) \Big] dt \\-B_k \Big[L_1^k (\overline{x_i-\hat{x}_{i|\mathcal{F}^y_i}})^k + L_2^k(\overline{x_0-\hat{x}_{0|\mathcal{F}^y_i}})^k + \sum_{l=1}^{K} \bar{L}_3^{k,l} \big(\overline{\bar{x}^{l}-\hat{\bar{x}}_{|\mathcal{F}^y_i}^{l} }\big)^k \allowdisplaybreaks\\+  L_4^k \big(\overline{\hat{x}_{0|\mathcal{F}^y_0}-(\hat{x}_{0|\mathcal{F}^y_0})_{|\mathcal{F}^y_i}}\big)^k+  \sum_{l=1}^{K} \bar{L}_5^{k,l}\big(\overline{\hat{\bar{x}}_{|\mathcal{F}^y_0}^{l}-(\hat{\bar{x}}_{|\mathcal{F}^y_0}^{l})_{|\mathcal{F}^y_i}}\big)^k \Big]dt, \end{multline}
where the overline symbol with superscript $k$, i.e. $\overline{(.)}^k,$ denotes the infinite-population limit of the average over subpopulation $k$ of the corresponding terms, which are the components of $\tilde{x}^{k, ex}_i$ in \eqref{EstDecompCompact} (see \textit{Proposition 3.1} in \cite{KizilkaleTAC2016} for the convergence analysis in q.m.). A compact representation of \eqref{HeuristicAveEstErrLimit} shall be used as in
\begin{multline}\label{MFmidProcess}
d\bar{x}^k  = \Big[(A_k + B_k L_1^k) \bar{x}^k + (G+B_k L_2^k) x_0+ B_k \Big(\sum_{l=1}^{K} \bar{L}_3^{k,l} \bar{x}^l + L_4^k \hat{x}_{0|\mathcal{F}^y_0}\allowdisplaybreaks\\+ \sum_{l=1}^{K} \bar{L}_5^{k,l} \hat{\bar{x}}_{|\mathcal{F}^y_0}^l + m_k \Big) \Big] dt + \bar{J}_k \bar{\tilde{x}}^{k,ex}dt, 
\end{multline}
where we denote by $\bar{\tilde{x}}^{k,ex}$ the average of the estimation errors of the minor agents of subpopulation $k$ as $N \rightarrow \infty$. Hence, the second bracket in \eqref{HeuristicAveEstErrLimit} is given by $\bar{J}_k \bar{\tilde{x}}^{k,ex}$ (Here the term  $\bar{J}_k \bar{\tilde{x}}^{k,ex}$ corrects its omission in \cite{KizilkaleTAC2016}.). In Section \ref{subsec:MFeq} we will derive the dynamical equation \eqref{EstErrLimit} that  $\bar{\tilde{x}}^{k,ex}$ satisfies. 

Therefore the mean field state vector $\bar{x}$ satisfies 
\begin{equation}\label{MeanFieldEq}
d\bar{x} = \left(\bar{A} \bar{x} + \bar{G} x_0 + \bar{H} \hat{x}_{0|\mathcal{F}^y_0}  + \bar{L} \hat{\bar{x}}_{|\mathcal{F}^y_0}  + \bar{J}\, \bar{\tilde{x}}^{ex} + \bar{m}\right)dt , 
\end{equation}
where $(\bar{\tilde{x}}^{ex})^{T} = [(\bar{\tilde{x}}^{1,ex})^{T}, \dots, (\bar{\tilde{x}}^{K,ex})^{T}]$, and the matrices $\bar{A}$, $\bar{G}$, $\bar{H}$, $\bar{L}$, $\bar{J}$, and $\bar{m}$ collect the corresponding terms in \eqref{MFmidProcess} and have the block matrix form
\begin{gather}
\bar{A} =  \begin{bmatrix}
\bar{A}_1\\
\vdots \\
\bar{A}_K \end{bmatrix} , \quad
\bar{G} = \begin{bmatrix}
\bar{G}_1\\
\vdots \\
\bar{G}_K \end{bmatrix}, \quad 
\bar{H} = \begin{bmatrix}
\bar{H}_1\\
\vdots \\
\bar{H}_K \end{bmatrix}, 
\nonumber \allowdisplaybreaks\\ 
\bar{L} =  \begin{bmatrix}
\bar{L}_1\\
\vdots \\
\bar{L}_K \end{bmatrix}, \quad
\bar{m} = \begin{bmatrix}
\bar{m}_1\\
\vdots \\
\bar{m}_K\end{bmatrix} ,\quad
\bar{J} =  \begin{bmatrix}
\bar{J}_1& & 0\\
&\ddots &\\
0& &\bar{J}_K \end{bmatrix}.\label{MFcoeffUn}
\end{gather}
%We note that $\bar{A}_k,\, \bar{L}_k \in \mb{R}^{n \times nK},\, \bar{G}_k,\,\bar{H}_k \in \mb{R}^{n \times n},\, \bar{m}_k\in \mb{R}^{n},\, \bar{J}_k \in \mb{R}^{n \times (3n+2nK)},\, 1 \leq k \leq K$, are to be solved for using the consistency equations \eqref{mfeqConsistency} that will be derived in Section \ref{subsec:MFeq}. 
with $\bar{A}_k,\, \bar{L}_k \in \mb{R}^{n \times nK},\, \bar{G}_k,\,\bar{H}_k \in \mb{R}^{n \times n},\, \bar{m}_k\in \mb{R}^{n},\, \bar{J}_k \in \mb{R}^{n \times (3n+2nK)},\, 1 \leq k \leq K$.  We use the heuristic mean field equation \eqref{MeanFieldEq} to formulate the stochastic optimal control problems for the agents in the infinite population limit. In Section \ref{subsec:MFeq}, the mean field equation parameters \eqref{MFcoeffUn} are obtained through the consistency equations \eqref{mfeqConsistency}, which effectively equate \eqref{MeanFieldEq} with the mean field resulting from the collective action of the mass of agents.  %are to be solved for using the consistency equations \eqref{mfeqConsistency} that will be derived in Section \ref{subsec:MFeq}. 
%In the end, the mean field equation parameters are obtained through the  consistency equations, these effectively equate the nominal  mean field employed in the initial form of the individual agents? control laws with the mean field resulting from the collective action of the mass of agents using these laws. 

By abuse of language, the mean value of the system's Gaussian mean field given by the state process $(\bar{x})^T = [(\bar{x}^1)^T, ...,(\bar{x}^K)^T]$ shall also be termed the system's mean field.% (The derivation of the properties above may performed using the methods of  \cite{KizilkaleTAC2016}, \cite{Kizilkale2013} and \cite{Huang2010}).
\subsection{Major Agent: Infinite Population}
The major agent's infinite population dynamics, as the number of agents goes to infinity ($N\rightarrow \infty$), remain the same as in \eqref{MajorAgentEq}, while its infinite population individual cost functional is given by
 \begin{gather} \label{MinorAgtCostInfPop}
J_{0}^{\infty} (u_0, u_{-0}) = \mathbb{E} \int_{0}^{\infty} e^{-\rho t} \Big \lbrace \Vert x_0 - \phi(\bar{x}) \Vert^{2}_{Q_0} + \Vert u_0 \Vert_{R_0}^{2} \Big \rbrace dt, \\
\phi(\bar{x}) := H_0^{\pi} \bar{x} + \eta_0, \\
H_0^\pi =  \pi  \otimes H_0 := [\pi_1 H_0,\pi_2 H_0,..., \pi_K H_0],
\end{gather}
where $x^{(N)}$ in \eqref{MajorCostLrgPop} was replaced by its $L^2$ limit, i.e. the mean field $\bar{x}$.  

To solve the infinite population tracking problem for the major agent, its state is extended with the mean field process $\bar{x}$, where this is assumed to exist, i.e. $(x_0^{ex})^T:= \big[x_0^T,\, \bar{x}^T\big]$. %Let the major agent's partial observations equation \eqref{MajorPartialObs} be rewritten as 
%\begin{align} \label{MajorObsExt}
%dy_0(t) = & \mathbb{L}_0 
%\left[ 
%x^T_0,
%\bar{x}^T
%\right]^T dt + R_{v_0}^{\frac{1}{2}} dv_0, \\
%&\mathbb{L}_0  = 
%\left[ \begin{array}{cc}
% l_0^1 &   0_{\ell \times nK}
%\end{array} \right].
%\end{align}

Then the Kalman filter which generates the estimates of the major agent's state $\hat{x}_{0|\mathcal{F}^y_0}$ and the mean field $\hat{\bar{x}}_{|\mathcal{F}^y_0}$ based on its own observations are, respectively, given by 
\begin{gather}
d\hat{x}_{0|\mathcal{F}^y_0} = A_0 \hat{x}_{0|\mathcal{F}^y_0} dt + B_0 \hat{u}_0 dt + K_0^1  d\nu_0,\\
d\hat{\bar{x}}_{|\mathcal{F}^y_0} = (\bar{G} + \bar{H}) \hat{x}_{0|\mathcal{F}^y_0} dt  + (\bar{A} + \bar{L}) \hat{\bar{x}}_{|\mathcal{F}^y_0} dt + \bar{m} dt + K_0^2 d\nu_0,  
\end{gather}
where $\hat{\bar{\tilde{x}}}_{|\mc{F}^y_0}=0$ is used (see \textit{Observation \ref{xTildeEst}}). Moreover, $\bar{m}$ is a deterministic process according to \eqref{MFmidProcess}, $K^1_0$ and $K^2_0$ are the Kalman filter gains, and $\nu_0$ is the innovation process. Therefore the Kalman filter which generates the estimates of the major agent's extended state is given by 
\begin{align}\label{majorKalman}
\left[ \begin{array}{c}
d\hat{x}_{0|\mathcal{F}^y_0}\\
d\hat{\bar{x}}_{|\mathcal{F}^y_0}
\end{array} \right] =  \left[ \begin{array}{cc}
A_0 & 0_{n\times nK}\\
\bar{G} + \bar{H} & \bar{A} + \bar{L}
\end{array} \right] \left[ \begin{array}{c}
\hat{x}_{0|\mathcal{F}^y_0}\\
\hat{\bar{x}}_{|\mathcal{F}^y_0}
\end{array} \right] dt + \left[ \begin{array}{c} 
B_0\\
0_{nK \times m}
\end{array}  \right] \hat{u}
_0 dt \nonumber \allowdisplaybreaks\\+
\left[ \begin{array}{c} 
0_{n \times 1}\\
\bar{m}
\end{array}  \right] dt + K_0 d\nu_0,
\end{align}
with the corresponding Kalman filter gain $K_0 = [(K_0^1)^T, (K_0^2)^T]^T$, and the innovation process $\nu_0$, respectively, given by
\begin{gather} \label{KalmanGainMajor}
K_0 = V_0 \mathbb{L}_0^T R_{v_0}^{- 1},\allowdisplaybreaks\\  \label{InnovationMajor}
 d\nu_0  =  dy_0 - \mathbb{L}_0 \Big[
\hat{x}^T_{0|\mathcal{F}_0^y},
\hat{\bar{x}}^T_{|\mathcal{F}_0^y}
\Big]^Tdt,
\end{gather}
where $\mathbb{L}_0  = \left[ \begin{array}{cc}  l_0^1 & 0_{p \times nK}
\end{array} \right]$, and 
$V_0(t)$ is the solution to the corresponding Riccati equation \eqref{RiccatiEqMajor}. 

From \eqref{MajorAgentEq}, \eqref{MeanFieldEq}, and \eqref{majorKalman} we denote
\begin{gather}
\mathbb{A}_0  = \left[ \begin{array}{cc}
A_0 & 0_{n\times nK} \\
\bar{G} + \bar{H} & \bar{A} + \bar{L}
\end{array} \right],~~~
\mathbb{B}_0 = \left[ \begin{array}{c}
B_0\\
0_{nK \times m} \end{array} \right],~~~
\mathbb{M}_0  = \left[ \begin{array}{c}
0_{n \times 1} \\
\bar{m}
\end{array} \right],\nonumber \allowdisplaybreaks\\\mathbb{D}_0 = \left[ 
\begin{array}{cc}
D_0 & 0_{n \times rK} \\
0_{nK \times r} & 0_{nK \times rK}
\end{array} \right],~~~\mb{J}_0 = \left[\begin{array}{c} 
0_{n \times (3nK+2nK^2)}\\
\bar{J}
\end{array}\right].  \label{M0D0} 
\end{gather}

Then to guarantee the convergence of the solution to the Riccati equation to a positive definite asymptotically stabilizing solution, we assume:
\begin{assumption} $[\mathbb{A}_0, \mathbb{D}_0]$ is stabilizable and $[\mathbb{L}_0, \mathbb{A}_0]$ is detectable.
\end{assumption}
The corresponding Riccati equation is then given by
\begin{equation} \label{RiccatiEqMajor}
\dot{V}_0  = \mathbb{A}_0 V_0 + V_0 \mathbb{A}_0^{T} - K_0R_{v_0}K_0^{T} + \mb{J}_0\bar{V}\mb{J}_0^T +Q_{w_0},
\end{equation}
where $Q_{w_0}=\mb{D}_0\mb{D}_0^T$, $\bar{V}(t) = \mb{E}\big[\bar{\tilde{x}}^{ex}(t)\big({\bar{\tilde{x}}^{ex}}(t)\big)^T \big]$ satisfies \eqref{Vbar}, and $V(0) = \mathbb{E} \big [\big(x_0^{ex}(0) - (\widehat{x _0^{ex}(0)})_{|\mathcal{F}^y_0} \big) \big( x_0^{ex}(0) - (\widehat{{x} _0^{ex}(0)})_{|\mathcal{F}^y_0} \big )^T\big]$.

 Then, utilizing the infinite horizon discounted analogy to \textit{\Cref{thm:StochIndefiniteLQ}}, it can be shown (see \textit{\Cref{Thm: POLQGMM-MFG}} in \Cref{subsec:MFeq}) that the optimal control action for the major agent's tracking problem (and hence best response MFG control input) is  
\begin{equation} \label{UOptMajEst}
\hat{u}_0^\circ = -R_0^{- 1} \mathbb{B}_0^T [\Pi_0 (\hat{x}_{0|\mathcal{F}^y_0}^{T},  \hat{\bar{x}}_{|\mathcal{F}^y_0}^{T})^T + s_0],
\end{equation} 
where $\Pi_0$ and $s_0$ are the solutions to the Riccati and offset equations given by 
\begin{gather}
\rho \Pi_0  =  \Pi_0 \mathbb{A}_0 + \mathbb{A}_0^T \Pi_0 - \Pi_0 \mathbb{B}_0 R_0^{-1} \mathbb{B}_0^T \Pi_0 + Q_0^\pi, \label{majorRiccati}\\
\rho s_{0} = \frac{ds_0}{dt} + (\mathbb{A}_0 - \mathbb{B}_0 R_0^{-1} \mathbb{B}_0^T \Pi_0 )^T s_0 + \Pi_0 \mathbb{M}_0 - \bar{\eta}_0,\label{majorOffset}
\end{gather}
with $\bar{\eta}_0 = [I_{n \times n}, -H_0^\pi]^TQ_0\eta_0$ and $Q_0^\pi = [I_{n \times n}, -H_0^\pi]^T Q_0 [I_{n \times n}, -H_0^\pi]$. We note $\tfrac{ds_0}{dt}=0$ in \eqref{majorOffset}, since $\mathbb{M}_0,\, \bar{\eta}_0$ are constant.

%\begin{align}
%V(0) = \mathbb{E} \Big [x_0^{\bar{x}}(0) - (\widehat{x _0^{\bar{x}}(0)})_{|\mathcal{F}^y_0} \Big ] \Big [x_0^{\bar{x}}(0) - (\widehat{{x} _0^{\bar{x}}(0)})_{|\mathcal{F}^y_0} \Big ]^T .
%\end{align}

%Q_0^\pi & = [I_{n \times n}, -H_0^\pi]^T Q_0 [I_{n \times n}, -H_0^\pi]
%\begin{align}
%&\bar{\eta}_0 = [I_{n \times n}, -H_0^\pi]^TQ_0\eta_0,\nonumber \\
%Q_0^\pi & = [I_{n \times n}, -H_0^\pi]^T Q_0 [I_{n \times n}, -H_0^\pi], \nonumber
%\end{align}
Finally, the joint dynamics of the major agent's closed-loop system and its Kalman filter system are given by
\begin{equation} \label{MajorClosedLoop}
\begin{bmatrix}
dx_0^{ex} \\
d\hat{x}_{0|\mc{F}^y_0}^{ex} \end{bmatrix}
= \Big(\mathbf{A}_0  \begin{bmatrix}
x_0^{ex} \\
\hat{x}_{0|\mc{F}^y_0}^{ex}
 \end{bmatrix} + \mathbf{J}_0\bar{\tilde{x}}^{ex}+ \mathbf{M}_0\!\Big) dt + \mathbf{D}_0
\begin{bmatrix}
\begin{bmatrix}
dw_0\\
0_{rK \times 1}\end{bmatrix} \\
d\nu_0
\end{bmatrix}, 
\end{equation}
where
\begin{gather*}
\mathbf{A}_0 = \left[ \begin{array}{cc}
\left[ \begin{array}{cc} 
A_0 & 0_{n \times nK} \\
\bar{G} &  \bar{A}
\end{array} \right]  &  \left[ \begin{array}{c}
 -B_0 R_0^{-1} \mathbb{B}_0^T \Pi_0\\
\left[ \begin{array}{cc}
\bar{H} & \bar{L} \end{array} \right] \end{array} \right], \\
0_{(n+nK)\times(n+nK)}  & \mathbb{A}_0 -\mathbb{B}_0 R_0^{-1} \mathbb{B}_0^T \Pi_0
\end{array} \right],\quad \mathbf{J}_0 = \left[ \begin{array}{c}
\left[\begin{array}{c}
0_{n \times (3nK+2nK^2)}\\
\bar{J}\end{array}\right]\\
0_{(n+nK)\times (3nK+2nK^2)}
 \end{array}\right], 
 \nonumber\\
\mathbf{M}_0 =  \left[\hspace{-1mm} \begin{array}{c}
\mathbb{M}_0 -\mathbb{B}_0 R_0^{-1} \mathbb{B}_0^T s_0\\
\mathbb{M}_0 -\mathbb{B}_0 R_0^{-1} \mathbb{B}_0^T s_0
\end{array} \hspace{-1mm}\right], \,\,\,\,
\mathbf{D}_0 = \left[ \hspace{-1mm}\begin{array}{cc}
\mathbb{D}_0 & 0_{(n+nK)\times p} \\
 0_{(n+nK)\times(r+rK)} & K_0
\end{array} \hspace{-1mm} \right]. 
\end{gather*}
%\left[ \begin{array}{cc} 
%A_0 & 0_{n \times nK} \\
%\bar{A} & \bar{G} 
%\end{array} \right] & \left[ \begin{array}{c}
% -B_0 R_0^{-1} \mathbb{B}_0^T \Pi_0\\
% [\bar{H}, \bar{L}] 
% \end{array} \right] \\
%[0_{n + nK \times n} , 0_{(n + nK) \times nK}] & \mathbb{A}_0 -B_0 R_0^{-1} \mathbb{B}_0^T \Pi_0

%\left[ \begin{array}{cc} 
%A_0 & 0_{n \times nK} \\
%\bar{A} & \bar{G} 
%\end{array} \right] & \left[ \begin{array}{c}
% -B_0 R_0^{-1} \mathbb{B}_0^T \Pi_0\\
% [\bar{H}, \bar{L}] 
% \end{array} \right] \\
%[0_{n + nK \times n} , 0_{(n + nK) \times nK}] & \mathbb{A}_0 -B_0 R_0^{-1} \mathbb{B}_0^T \Pi_0

\subsection{Minor Agent: Infinite Population}
A generic minor agent's infinite population dynamics, as the number of agent goes to infinity ($N\rightarrow \infty$), remain the same as in \eqref{MinorAgentEq}, while its  infinite population individual cost functional is given as
\begin{gather}
J_{i}^{\infty}(u_{i}, u_{0}) = \mathbb{E} \int_{0}^{\infty}  e^{-\rho t} \Big\{ \Vert x_{i} - \mathrm{\psi}(\bar{x},x_0)\Vert _{Q}^2 + \Vert u_i \Vert_{R}^2 \Big\} dt,\allowdisplaybreaks\\
\mathrm{\psi}(\bar{x},x_0) = H_1 x_0 + H_2^{\pi} \bar{x} + \eta,\allowdisplaybreaks\\
H_{2}^{\pi} = \pi \otimes H_2 := [\pi_1 H_2, \pi_2 H_2, ..., \pi_K H_2].\allowdisplaybreaks
\end{gather} 
In the case where all agents have partial observations on the major agent's state, the joint dynamics of the major agent's closed-loop system and its Kalman filtering recursions are employed in order to solve the minor agent's tracking problem. Before proceeding we enunciate \textit{Proposition 1}, where for ease of exposition, the simple case where it is assumed that the major agent and minor agents are not coupled with the mean field (neither in their dynamics nor in their cost functional) is considered. However, each minor agent is assumed to be coupled with the major agent's state in their cost functional. The results are  extendable to the more general case described by \eqref{MajorAgentEq}-\eqref{MinorAgentEq} and \eqref{MajorCostLrgPop}-\eqref{MinorCostLrgPop}, in a straightforward way.

\begin{proposition}(Estimates of Estimates Filter) \label{props: EstofEst}
Let the major agent's dynamics be given by
\begin{align}
 dx_0 = A_0 x_0 dt + B u_0 dt + dw_0,
\end{align}
 and the major agent's observations of its own state by 
\begin{align}
dy_0 = H_0 x_0dt + dv_0,
\end{align}
 then the estimates of the major agent's state based on its own observation is generated by 
\begin{align} \label{LemKF}
&d\hat{x}_{0|\mathcal{F}^y_0} = A_0 \hat{x}_{0|\mathcal{F}^y_0}dt + B \hat{u}_0 dt+ K_0 [dy_0 - H_0 \hat{x}_0 dt] := \nonumber\\ 
&A_0 \hat{x}_{0|\mathcal{F}^y_0} dt + B \hat{u}_0 dt + K_0 [H_0 x_0 dt + dv_0 - H_0 \hat{x}_{0|\mathcal{F}^y_0}dt].
\end{align}
Next, assume the major agent's control action is of the form
\begin{align}
u_0 = -L\hat{x}_{0|\mathcal{F}^y_0},
\end{align}
then in this case the joint dynamics of the major agent's closed-loop system and its Kalman filter system are given by
\begin{align} \label{MajorJoint}
\left[ \begin{array}{c}
dx_0\\
d\hat{x}_{0|\mathcal{F}^y_0} 
\end{array} \right] =
\left[ \begin{array}{cc}
A_0 & -BL\\
K_0 H_0 & A_0-BL-K_0 H_0 
\end{array} \right] 
& \left[ \begin{array}{c}
x_0\\
\hat{x}_{0|\mathcal{F}^y_0} 
\end{array} \right]  dt  \nonumber\\
+ \left[ \begin{array}{c}
0\\
K_0 dv_0
\end{array} \right] + 
&\left[ \begin{array}{c}
dw_0\\
0
\end{array} \right].
\end{align}
Finally, let the minor agent's partial observations of the major agent's state be given by
\begin{align} \label{LemMinorObs}
dy_i  =  H_i x_0dt + dv_i
& = \left[ \begin{array}{cc}
H_i & 0
\end{array} \right] 
\left[ \begin{array}{c}
x_0\\
\hat{x}_{0|\mathcal{F}^y_0} 
\end{array} \right]dt + dv_i,
\end{align}
then the process of estimates of the state of \eqref{MajorJoint} based upon the observations \eqref{LemMinorObs}  is generated by the filtering scheme
\begin{align} \label{MajorJointEst}
&\hspace{-4mm} \left[ \begin{array}{c}
\hspace{-1mm} d\hat{x}_{0|\mathcal{F}^y_i}\hspace{-1mm}\\
\hspace{-1mm} d(\hat{x}_{0|\mathcal{F}^y_0})_{|\mathcal{F}^y_i} \hspace{-1mm}
\end{array} \right] \hspace{-1mm} = \hspace{-1mm}
\left[ \begin{array}{cc}
\hspace{-1mm} A_0 & -BL \hspace{-1mm}\\
\hspace{-1mm} K_0 H_0 & A_0-BL-K_0 H_0 \hspace{-1mm}
\end{array} \right]  \hspace{-1mm}
\left[ \begin{array}{c}
\hspace{-1mm} \hat{x}_{0|\mathcal{F}^y_i} \hspace{-1mm}\\
\hspace{-1mm} (\hat{x}_{0|\mathcal{F}^y_0})_{|\mathcal{F}^y_i} \hspace{-1mm} 
\end{array} \right]dt  \nonumber \\
&\quad \quad + K_i \bigg ( dy_i- 
\left[ \begin{array}{cc}
H_i & 0
\end{array} \right]
\left[ \begin{array}{c}
d\hat{x}_{0|\mathcal{F}^y_i}\\
d(\hat{x}_{0|\mathcal{F}^y_0})_{|\mathcal{F}^y_i} 
\end{array} \right]dt \bigg),
\end{align}
where $\hat{x}_{0|\mathcal{F}^y_i} := \mathbb{E}_{|\mathcal{F}^y_i} x_0 = \mathbb{E} \{ x_0| \mathcal{F}^y_i\}$ denotes the minor agent $\mathcal{A}_i$'s estimate of the major agent's state, and 
$$(\hat{x}_{0|\mathcal{F}^y_0})_{|\mathcal{F}^y_i} := \mathbb{E}_{| \mathcal{F}^y_i} \hat{x}_{0|\mathcal{F}^y_0} =\mathbb{E} \{ \hat{x}_{0|\mathcal{F}^y_0} |  \mathcal{F}^y_i  \},$$
denotes the minor agent $\mathcal{A}_i$'s estimate of the major agent's estimate of its own state.
\hfill $\square$
\end{proposition}
 The proof of the Proposition \ref{props: EstofEst} is straightforward and will be omitted, but we observe that the key property of the overall system which ensures its validity is that the Wiener processes $\{w_0 ,v_0\}$ are independent of the noise process $v_i$ in \eqref{MajorJointEst}.

 Returning to the main problem, the minor agent's state is next extended to form $(x_i^{ex})^T := \big[x_i^T,~ x_0^T, ~\bar{x}^T,~\hat{x}_{0|\mathcal{F}^y_0}^T,~ \hat{\bar{x}}_{|\mathcal{F}^y_0}^T\big]$. Specifically this yields
\begin{equation} \label{minorSys_noUhat}
dx_i^{ex} = \left(\mathbb{A}_k x_i^{ex} +\mathbb{B}_k u_i + \mathbb{J} \bar{\tilde{x}}^{ex}+  \mathbb{M}\right) dt + \mathbb{D} [dw_i^T, dw_0^T, 0_{1\times rK}, d\nu_0^T]^T,
\end{equation}

where
 \begin{gather}
\mathbb{A}_k  =  \left[\hspace{-1mm}
\begin{array}{ccc}
A_k & [\begin{array}{cc}
G & 0_{n \times (n+2nK)}
\end{array}]  \\
0_{2(n+nK)\times n} & \mathbf{A}_0
\end{array} \hspace{-1mm}\right], \quad \mathbb{B}_k =  \left[\hspace{-1mm} \begin{array}{c}
B_k\\
0_{2(n+nK) \times m}
\end{array} \hspace{-1mm}\right], \nonumber \allowdisplaybreaks\\
\mathbb{J} = \left[\begin{array}{c}
0_{n \times (3nK+2nK^2)}\\
\mathbf{J}_0 
\end{array}\right], \quad
\mathbb{M} = \left[ \begin{array}{c} 
0_{n\times 1} \\
\mathbf{M}_0
\end{array} \right], \quad \mathbb{D} =  
\left[ \begin{array}{ccc}
D & 0_{n \times (r+rK+p)}\\
0_{2(n+nK) \times r}  & \mathbf{ D}_0
\end{array} \right]. 
\end{gather}
To derive the Kalman filter equations for \eqref{minorSys_noUhat}, we first define $\mb{L}_k = \begin{bmatrix}l_k^1& l_k^2 & 0_{p \times (n+2nK)}\end{bmatrix}$.
%Denote the minor agent's system noise covariance matrix by 
%$Q_{w}^{k}= \mathbb{D} \mathbb{D}^T$
\begin{comment}
\begin{align*}
Q_{w}^{k}& = \mathbf{D} \mathbf{D}^T
 = \left[ \begin{array}{cccc}
\Sigma_{k} & 0 & 0 & 0\\
0 & \Sigma_0 & 0 & 0\\
0 & 0 & 0 & 0\\
0 & 0 & 0 & \hat{\Sigma}
\end{array} \right],
\end{align*}
\end{comment}
%and its measurement noise covariance matrix by $R_v$. 
To guarantee the convergence of the solution to the Riccati equation to a positive definite asymptotically stabilizing solution, we assume: 
\begin{assumption} \label{ass: RiccatiMinor}
The system parameter set $\Theta = \{1,...,K\}$ is such that $[\mathbb{A}_k, \mathbb{D}]$ is stabilizable and $[\mathbb{L}_{k}, \mathbb{A}_k]$ is detectable for all $k$, $1 \leq k \leq K$.
\end{assumption}

The Riccati equation associated with the filtering equations for \eqref{minorSys_noUhat} is then given by
\begin{equation} \label{RiccatiEqMinor}
\dot{V}_i = \mathbb{A}_k V_i + V_i \mathbb{A}_k^T - K_i R_v K_i^T +\mb{J}\bar{V}\mb{J}^T + Q_{w},
\end{equation}
where $Q_{w}= \mathbb{D} \mathbb{D}^T$, $\bar{V}(t) = \mb{E}\big[\bar{\tilde{x}}^{ex}(t)\big({\bar{\tilde{x}}^{ex}}(t)\big)^T \big]$ satisfies \eqref{Vbar}, and $V_k(0) = \mathbb{E} \big [\big( x_i^{ex}(0) - (\widehat{x_i^{ex}(0)})_{|\mathcal{F}^y_i}\big)\big( x_i^{ex}(0) - (\widehat{x_i^{ex}(0)})_{|\mathcal{F}^y_i}\big)^T\big ]$. The Kalman filter gain $K_i$ is in turn given by 
\begin{equation}
K_i  =  V_i \mathbb{L}_{k}^{T} R_{v}^{- 1},
\end{equation}
and the innovation process $\nu_i(t)$ is defined as in
\begin{equation} \label{MinorInnovation}
d\nu_i  \hspace{-0.5mm} = \hspace{-0.5mm} dy_i - \mathbb{L}_{k} \hspace{-1mm}\left[
\hat{x}^T_{ i| \mathcal{F}^y_i}, 
\hat{x}^T_{ 0| \mathcal{F}^y_i},
\hat{\bar{x}}^T_{|\mathcal{F}^y_i},
(\hat{x}_{0|\mathcal{F}^y_0})_{|\mathcal{F}^y_i}^T,
(\hat{\bar{x}}_{|\mathcal{F}^y_0})_{|\mathcal{F}^y_i}^T
\right]^T \hspace{-2mm} dt,
\end{equation}
where $(\hat{x}_{0|\mathcal{F}^y_0})_{|\mathcal{F}^y_i}$ and 
$(\hat{\bar{x}}_{|\mathcal{F}^y_0})_{|\mathcal{F}^y_i}$, respectively, denote the minor agent $\mc{A}_i$'s estimates of the major agent's estimates of its own state and the mean field. Then the Kalman filter equations for a generic minor agent $\mathcal{A}_i, \, 1 \leq i \leq N$, are given as in
\begin{equation} \label{minorKalman}
d\hat{x}^{ex}_{i|\mathcal{F}^y_i} = \mathbb{A}_k \hat{x}^{ex}_{i|\mathcal{F}^y_i} dt + \mathbb{B}_k \hat{u}_i dt + \mathbb{M} dt + K_i d\nu_i,
\end{equation} 
where $\hat{\bar{\tilde{x}}}^{ex}_{|\mc{F}^y_i}=0$ (see \textit{Observation \ref{xTildeEst}}) is used.
Clearly, \eqref{minorKalman} generates the iterated estimates $(\hat{x}_{0|\mathcal{F}^y_0})_{|\mathcal{F}^y_i}$ and $(\hat{\bar{x}}_{|\mathcal{F}^y_0})_{|\mathcal{F}^y_i}$ which are required to calculate $\hat{x}_{0|\mathcal{F}^y_i}$ and $\hat{\bar{x}}_{|\mathcal{F}^y_i}$ (see \textit{Proposition \ref{props: EstofEst}} for a simplified case of Estimates of Estimates Filter). 
%Before proceeding we enunciate \textit{Proposition 1}, where for ease of exposition, the simple case where it is assumed that the major agent and minor agents are not coupled with the mean field (neither in their dynamics nor in their cost functional) is considered. However, each minor agent is assumed to be coupled with the major agent's state in their cost functional. The results are  extendable to the more general case described by \eqref{MajorAgentEq}-\eqref{MinorAgentEq} and \eqref{MajorCostLrgPop}-\eqref{MinorCostLrgPop}, in a straightforward way.
\begin{remark}
By virtue of the asymmetric information available to the major agent and a generic minor agent, an infinite regress does not occur in the process of estimating other agents' states. In fact to calculate the best response action, the major agent only estimates its own state and hence does not estimate minor agents' states, while each minor agent estimates its own state and the major agent's state.
\end{remark}

We note that by  \textit{Assumption \ref{ass:MajorControl}} the minor agent $\mathcal{A}_i$ is able to estimate $\hat{u}_0^\circ$ whenever the functional dependence of the major agent's control on it's state is available to the minor agent through forming the conditional expectation of the major agent's control action which by \eqref{UOptMajEst} is given by the following expression    
\begin{equation} \label{Uhathat}
(\hat{u}_0^\circ)_{|\mathcal{F}^y_i} := \mathbb{E} \lbrace \hat{u}_0^\circ | \mathcal{F}^y_i\rbrace = 
 -R^{-1} \mathbb{B}_0^T \Big [\Pi_0 \Big( (\hat{x}_{0| \mathcal{F}^y_0})_{|\mathcal{F}^y_i}^T, 
(\hat{\bar{x}}_{|\mathcal{F}^y_0})_{|\mathcal{F}^y_i}^T\Big) + s_0 \Big],
\end{equation}
and which is embedded in \eqref{minorKalman}. 
Then, utilizing the infinite horizon discounted analogy to  \textit{\Cref{thm:StochIndefiniteLQ}}, it can be shown (see \textit{\Cref{Thm: POLQGMM-MFG}}) that the optimal control action for the minor agent $\mathcal{A}_i$'s tracking problem (and hence best response MFG control input) is given by
\begin{equation}\label{minorCntrl}
\hat{u}_i^\circ = -R^{-1} \mathbb{B}_k^T 
 \left [\Pi_k \Big( \hat{x}_{i|\mathcal{F}_i^y}^T,\hat{x}_{0|\mathcal{F}_i^y}^T,\hat{\bar{x}}_{|\mathcal{F}_i^y}^T, (\hat{x}_{0|\mathcal{F}^y_0})_{|\mathcal{F}^y_i}^T,
 (\hat{\bar{x}}_{|\mathcal{F}^y_0})_{|\mathcal{F}^y_i}^T \Big)^T+s_k\right],
\end{equation}
where the iterated estimation terms $(\hat{x}_{0|\mathcal{F}^y_0})_{|\mathcal{F}^y_i}$, and $(\hat{\bar{x}}_{|\mathcal{F}^y_0})_{|\mathcal{F}^y_i}$  explicitly appear, and the corresponding Riccati and offset equations are given by 
\begin{gather}
\rho \Pi_k  = \Pi_k \mathbb{A}_k + \mathbb{A}_k^T \Pi_k - \Pi_k \mathbb{B}_k R^{-1} \mathbb{B}_k^T \Pi_k + Q^{\pi},~ \forall k, \label{minorRiccati}\allowdisplaybreaks\\
\rho s_{k} = \frac{ds_k}{dt} + (\mathbb{A}_k - \mathbb{B}_k R^{-1} \mathbb{B}_k^T \Pi_k )^T s_k + \Pi_k \mathbb{M} - \bar{\eta},~ \forall k, \label{minorOffset}
\end{gather}
with $\bar{\eta} = [I_{n \times n},\,-H_1,\,-H_2^\pi,\,0_{n \times (n+nK)} ]^T Q \eta$, and $Q^{\pi} = [I_{n\times n},\,-H_1,\, -H_2^{\pi},\, 0_{n \times (n+nK)}]^T Q [I_{n \times n},\,-H_1,\,-H_2^{\pi},\, 0_{n \times (n+nK)}]$. We note $\tfrac{ds_k}{dt}=0$ in \eqref{minorOffset}, since $\mathbb{M},\, \bar{\eta}$ are constant. 

%Finally \eqref{minorCntrl} is substituted in \eqref{minorSys_noUhat} to form the closed-loop dynamics for the minor agent's extended system whose solution is given by 
%\begin{multline}\label{minorExtDynSol}
%x_i^{ex}(t) = \Phi_k(t,0) x_i^{ex}(0) + \int_{0}^t \Phi_k(t,\tau)\mathbb{M}(\tau)d\tau +  \int_{0}^t \Phi_k(t,\tau) \hat{u}^{\circ}_i(\tau)d\tau \\+\int_{0}^t \Phi_k(t,\tau) \mathbb{D} [dw_i^T(\tau), dw_0^T(\tau), 0_{nk\times1}^T, dv_0^T(\tau)]^T,
%\end{multline}
%where $\Phi_k(t,\tau)=\exp[\mb{A}_k(t-\tau)]$.
%\begin{align}
%&\bar{\eta} = [I_{n \times n},~-H,~-H_2^\pi ]^T Q \eta, \nonumber\\
%Q^{\pi} = &[I_{n\times n},~-H_1, -H_2^{\pi}]^T Q [I_{n \times n},~-H_1,~-H_2^{\pi}].\nonumber
%\end{align}
\subsection{Mean Field Consistency Equations} \label{subsec:MFeq}
Let us denote the components of $\Pi_k$ in \eqref{minorRiccati} as
\begin{align}
\Pi_k = \left[ \begin{array}{ccccc}\label{Pi_decomp}
\Pi_{k,11} & \Pi_{k,12} & \Pi_{k,13} & \Pi_{k,14} &  \Pi_{k,15}  \\
\Pi_{k,21} & \Pi_{k,22} & \Pi_{k,23} &  \Pi_{k,24} &  \Pi_{k,25}
%\Pi_{k,31} & \Pi_{k,32} & \Pi_{k,33} &  \Pi_{k,34} &  \Pi_{k,35}\\
%\Pi_{k,41} & \Pi_{k,42} & \Pi_{k,43} & \Pi_{k,44} & \Pi_{k,45} \\
%\Pi_{k,51} & \Pi_{k,52} & \Pi_{k,53} & \Pi_{k,54} & \Pi_{k,55}
\end{array} \right], 
\end{align}
$1 \leq k \leq K$, and where $\Pi_{k,11},\, \Pi_{k,12}, \, \Pi_{k,14} \in \mathbb{R}^{n \times n}$, $\Pi_{k,13},\, \Pi_{k,15} \in \mathbb{R}^{n \times nK}$, $\Pi_{k,21},\, \Pi_{k,22},\, \Pi_{k,24} \in \mathbb{R}^{2(n+nK) \times n}$, and $\Pi_{k,23}, \, \Pi_{k,25} \in \mathbb{R}^{2(n+nK) \times nK}$. Let us also define the block matrix $\mathbf{e}_{k,q} = [0_{q \times q}, ..., 0_{q \times q}, I_q, 0_{q \times q}, ..., 0_{q \times q}]$ with $K$ blocks, where the identity matrix $I_q$ is located at the $k$th block.  Finally we define the block matrix $\mathbf{1}_q=[I_{q}, ..., I_{q},...,I_{q}]$ with $K$ blocks of identity matrix. Then we denote by 
\begin{gather}
\bar{\mathbf{e}}_{k} = \mathbf{e}_{k,n},\\
\tilde{\mathbf{e}}_{k} = \mathbf{e}_{k,(3n+2nK)},\\
\tilde{\mathbf{1}} = \mathbf{1}_{(3n+2nK)}
\end{gather}

To obtain the mean field consistency equations, we substitute \eqref{minorCntrl} in \eqref{MinorAgentEq} to get 
 \begin{equation}\label{minorClosedSys_allPO}
dx_i = A_k x_i dt+ G x_0 dt  - B_k R^{-1} \mathbb{B}_k^T 
 \big[\Pi_k \hat{x}_{i|\mc{F}^y_i}^{ex}+s_k\big] dt + D dw_i.\end{equation} 
Then $\hat{x}_{i|\mc{F}^y_i}^{ex}$ can be written as
\begin{align}
\hat{x}_{i|\mc{F}^y_i}^{ex} &= -(x_i^{ex}-\hat{x}_{i|\mc{F}^y_i}^{ex}) + x_i^{ex},\nonumber\\
&= -\tilde{x}^{ex}_i +  x_i^{ex},\label{stateDecomp}
\end{align}
where $\tilde{x}^{ex}_i$ denotes the estimation error, and the governing dynamics in the steady state for $1\leq i\leq N$, $1\leq k \leq K$, are given by 
\begin{equation}\label{estErrDyn}
d \tilde{x}^{k,ex}_i = (\mb{A}_k-K_k\mb{L}_k)\tilde{x}^{k,ex}_i+ \mathbb{J}\bar{\tilde{x}}^{ex} dt- K_k \sigma_vdv_i +\mb{D}[dw_i^T, dw_0^T, 0_{1\times rK}, d\nu_0^T]^T, 
\end{equation}
 where $(\bar{\tilde{x}}^{ex})^T= [(\bar{\tilde{x}}^{1,ex})^T, \dots,(\bar{\tilde{x}}^{K,ex})^T]$ satisfies \eqref{estErrVecDynCompact}.

  Next the empirical average of \eqref{minorClosedSys_allPO}, where \eqref{stateDecomp} has been substituted, over the population of the minor agents of type $k$ is given by 
  \begin{multline}
d(\frac{1}{N_k} \sum_{i\in \I_k} x_i^{k}) = D\frac{1}{N_k} \sum_{i\in \I_k}dw_i + A_k (\frac{1}{N_k} \sum_{i\in \I_k} x_i^{k}) dt  + G x_0 dt\\  \hspace{-1.4cm}- B_k R^{-1} \mathbb{B}_k^T \Big [\Pi_k \big(-\frac{1}{N_k} \sum_{i\in \I_k} \tilde{x}_{i}^{k,ex}+ \frac{1}{N_k} \sum_{i\in \I_k} {x}_{i}^{k,ex}\big) +s_k\Big] dt. \label{EmpAveClosedMinor_allPO}
\end{multline} 
As $N_k\rightarrow \infty$, the solution to \eqref{EmpAveClosedMinor_allPO} converges, in quadratic mean, to the solution of 
 \begin{equation}\label{EmpAveClosedMinor_allPO_limit}
d \bar{x}^{k} = A_k \bar{ x}^{k} dt  + G x_0 dt- B_k R^{-1} \mathbb{B}_k^T \Big [\Pi_k \big(- \bar{\tilde{x}}^{k,ex}+ \bar{x}^{k,ex}\big) +s_k\Big] dt,
\end{equation}
where $\bar{x}^{k,ex} = \big [(\bar{x}^k)^T, x_0^T, \bar{x}^T, \hat{x}_{0|\mc{F}^y_0}^T, \hat{\bar{x}}_{|\mc{F}^y_0}^T\big]^T$, and from \eqref{estErrDyn} $\bar{\tilde{x}}^{k,ex}$ (the average of estimation error over subpopulation $k$ as $N\rightarrow \infty$) satisfies  

\begin{equation}\label{EstErrLimit} 
d \bar{\tilde{x}}^{k,ex} = \big[(\mb{A}_k-K_k\mb{L}_k) \bar{\tilde{x}}^{k,ex} + \mathbb{J}\bar{\tilde{x}}^{ex}\big]dt +\mb{D}\big[0_{1\times r}, dw_0^T, 0_{1\times rK}, d\nu_0^T\big]^T.  
\end{equation}
Note that in the derivation of \eqref{EstErrLimit}, we use the property that $\frac{1}{N_k}\sum_{i\in \I_k} w_0 = w_0$ and $\frac{1}{N_k}\sum_{i\in \I_k}\nu_0 = \nu_0$, since $w_0$ and $\nu_0$ are the common processes shared between all agents of type $k$. Moreover, the law of large numbers is used to obtain as $N_k \rightarrow \infty$ 
\begin{gather*}
\frac{1}{N_k}\sum_{i\in \I_k}K_k d\nu_i \xrightarrow{\text{q.m.}} 0, \quad
\frac{1}{N_k}\sum_{i\in \I_k}dw_i \xrightarrow{\text{q.m.}} 0. \end{gather*}

Subsequently, from \eqref{EstErrLimit},  $(\bar{\tilde{x}}^{ex})^T= [(\bar{\tilde{x}}^{1,ex})^T, \dots,(\bar{\tilde{x}}^{K,ex})^T]$ satisfies
   \begin{equation}\label{estErrVecDyn}
d\bar{\tilde{x}}^{ex} = \begin{bmatrix}
(\mb{A}_1-K_1\mb{L}_1)\tilde{\mathbf{e}}_1+ \mathbb{J}\\
\vdots\\
(\mb{A}_K-K_k\mb{L}_k)\tilde{\mathbf{e}}_K+ \mathbb{J}\\
\end{bmatrix}\bar{\tilde{x}}^{ex} dt+ \begin{bmatrix}
\mb{D} \\
 \vdots  \\ 
 \mb{D}\end{bmatrix} \begin{bmatrix}
 0_{r \times 1}\\
  dw_0\\
 0_{rK\times 1}\\
 d\nu_0
 \end{bmatrix}, 
\end{equation}
or equivalently in the compact form 
\begin{equation}\label{estErrVecDynCompact}
d\bar{\tilde{x}}^{ex} = \tilde{\mb{A}}\bar{\tilde{x}}^{ex}dt + \tilde{\mb{D}}\big[0_{1\times r}, dw_0^T, 0_{1\times rK}, d\nu_0^T\big]^T.
\end{equation}

Using \eqref{Pi_decomp} the mean field equation \eqref{EmpAveClosedMinor_allPO_limit} can be presented as 
\begin{multline}
d \bar{x}^{k} = \Big( \big[A_k - B_k R^{-1} B_k^T \Pi_{k,11}\big] \bar{\mathbf{e}}_k - B_k R^{-1} B_k^T \Pi_{k,13}\Big) \bar{ x} dt  +\Big(G-B_k R^{-1} B_k^T \Pi_{k,12}\Big) x_0 dt \\-B_kR^{-1} \Big(B_k^T \Pi_{k,14}\hat{x}_{0|\mc{F}^y_0}+ B_k^T \Pi_{k,15}\hat{\bar{x}}_{|\mc{F}^y_0} - \mb{B}_k^T\Pi_k \bar{\tilde{x}}^{k,ex} + \mathbb{B}_k^{T} s_k\Big)dt.\label{MFeqAsObtained}
\end{multline}
%Since \eqref{EmpAveClosedMinor_allPO_limit} and \eqref{MeanFieldEq} must be identical, we obtain the Consistency Equations, 
In order to generate a mean-field game equilibrium \eqref{MFeqAsObtained} and the $k$th component of \eqref{MeanFieldEq} must correspond to the same dynamical  system generating the mean field. Consequently we obtain the Consistency Condition equations, determining the components of $\bar{A}$, $\bar{G}$, $\bar{H}$, $\bar{L}$, $\bar{J}$, and $\bar{m}$ in \eqref{MeanFieldEq}, given by the following compact set of equations
\begin{align}
%&\rho \Pi_0  =  \Pi_0 \mathbb{A}_0 + \mathbb{A}_0^T \Pi_0 - \Pi_0 \mathbb{B}_0 R_0^{-1} \mathbb{B}_0^T \Pi_0 + Q_0^\pi,\nonumber\\
%&\rho \Pi_k  = \Pi_k \mathbb{A}_k + \mathbb{A}_k^T \Pi_k - \Pi_k \mathbb{B}_k R^{-1} \mathbb{B}_k^T \Pi_k + Q^{\pi},\nonumber\\
&\bar{A}_k  = \big[A_k - B_k R^{-1} B_k^T \Pi_{k,11}\big] \bar{\mathbf{e}}_k - B_k R^{-1} B_k^T \Pi_{k,13}, \nonumber \allowdisplaybreaks\\
&\bar{G}_k  ~= G-B_k R^{-1} B_k^T \Pi_{k,12}, \nonumber \allowdisplaybreaks\\
& \bar{H}_k = -B_k R^{-1} B_k^T \Pi_{k,14}, \nonumber \allowdisplaybreaks\\
& \bar{L}_k = -B_k R^{-1} B_k^T \Pi_{k,15},    \nonumber \allowdisplaybreaks\\
& \bar{J}_k = B_k R^{-1}\mb{B}_k^T\Pi_k,  \nonumber \allowdisplaybreaks\\
%&\rho s_{0} = \tfrac{ds_0}{dt} + (\mathbb{A}_0 - \mathbb{B}_0 R_0^{-1} \mathbb{B}_0^T \Pi_0 )^T s_0 + \Pi_0 \mathbb{M}_0 - \bar{\eta}_0,\nonumber\\
%& \rho s_{k} = \tfrac{ds_k}{dt} + (\mathbb{A}_k - \mathbb{B}_k R^{-1} \mathbb{B}_k^T \Pi_k )^T s_k + \Pi_k \mathbb{M} - \bar{\eta},\nonumber\\
&\bar{m}_k ~= -B_k R^{-1} \mathbb{B}_k^{T} s_k,
\label{mfeqConsistency}
\end{align} 
 $1\leq k \leq K$, where $\Pi_k$ and $s_k$ satisfy \eqref{minorRiccati} and \eqref{minorOffset}, respectively. 
The set of equations \eqref{mfeqConsistency} together with \eqref{majorRiccati}-\eqref{majorOffset} and \eqref{minorRiccati}-\eqref{minorOffset} form a fixed point problem which must be solved by each individual agent $\mc{A}_i, 0 \leq i \leq N$, in order to compute the matrices in the mean field dynamics \eqref{MeanFieldEq}. 

Finally from \eqref{minorSys_noUhat} and \eqref{EmpAveClosedMinor_allPO_limit}-\eqref{estErrVecDynCompact} the Markovian dynamics of $\bar{x}^{k}$ (i.e. the mean field of subpopulation $k$, and the first component of $\bar{x}^{k,ex}$) are given by 
\begin{multline}
 \begin{bmatrix}
d\bar{x}^{k,ex} \\
d\bar{\tilde{x}}^{ex}
\end{bmatrix}= \begin{bmatrix}
\mb{A}_k-\mb{B}_kR^{-1}\mb{B}_k^T\Pi_k & \mathbb{J}+\mb{B}_kR^{-1}\mb{B}_k^T\Pi_k \tilde{\mathbf{e}}_k\\
0 & \tilde{\mb{A}}_k
\end{bmatrix} \begin{bmatrix}
\bar{x}^{k,ex} \\
\bar{\tilde{x}}^{ex}
\end{bmatrix} dt\\
+\begin{bmatrix}
\mb{M}-\mb{B}_k R^{-1} \mathbb{B}_k^{T} s_k\\
0
\end{bmatrix} dt +  \begin{bmatrix}
\mb{D} \\
\tilde{\mb{D}} 
\end{bmatrix}
\begin{bmatrix}
0_{r \times 1}\\
dw_0\\
0_{rK \times 1}\\
d\nu_0
\end{bmatrix}.
\end{multline} 

\begin{remark}
From \eqref{EstErrLimit}, in the infinite population limit, the average of the estimation errors of the minor agents of type $k, 1\leq k\leq K$, is driven by the major agent's Wiener process $w_0$ and the measurement noise $v_0$ (or equivalently the innovation process $\nu_0$). In other words, it is driven by the non-zero quadratic variation processes in the dynamics of the common processes $x_0^{ex}, \hat{x}_{0|\mc{F}^y_0}^{ex}$, with which the minor agents $\mc{A}_i, 1\leq i \leq N$, are coupled.

 From \eqref{estErrVecDynCompact}, $\bar{V}(t) = \mb{E}\big[\bar{\tilde{x}}^{ex}(t)\big({\bar{\tilde{x}}^{ex}}(t)\big)^T \big]$ satisfies 
%\begin{equation}
%\dot{\bar{V}} = \tilde{\mb{A}} \bar{V} + \bar{V} \tilde{\mb{A}}^T + \tilde{\mb{D}} \left[\begin{array}{cccc}
%0_{r \times r} & & &\\
%& I_{r \times r} & & \\
%& & 0_{rK \times rK}&\\
%&  & & I_{r\times r}
%\end{array}\right]\tilde{\mb{D}}^T,
%\end{equation}
\begin{equation}\label{Vbar}
\dot{\bar{V}} = \tilde{\mb{A}} \bar{V} + \bar{V} \tilde{\mb{A}}^T + \tilde{\mb{Q}} \tilde{\mb{Q}}^T,\end{equation}
where $\tilde{\mb{D}} = \tilde{\mathbf{1}}^T\mb{D}$ is used and
\begin{equation}
 \tilde{\mb{Q}} \tilde{\mb{Q}}^T = \tilde{\mathbf{1}}^T\mb{D}\begin{bmatrix}
0_{n \times n}  & 0 & 0 & 0\\
0 & I_n  & 0  & 0\\
0 & 0& 0_{nK \times nK}& 0\\
0 & 0 & 0 & \mb{L}_0 V_0 \mb{L}_0^T+R_{v_0}
\end{bmatrix}\mb{D}^T\tilde{\mathbf{1}}.
\end{equation}
To guarantee the convergence of the solution to the corresponding Lyapunov equation to a unique, symmetric and positive definite solution, we assume: 
\begin{assumption}\label{LyaCond} The pair $(\tilde{\mathbb{A}},  \tilde{\mb{Q}})$ is controllable.
 \hfill $\square$
\end{assumption}
\end{remark}

\begin{remark} 
For the case where the major agent has complete observations on its own state, and each minor agent has complete observations on their own state and the major agent's state, we have (i) $\mb{E}\{x_0|\mc{F}^y_0\} =x_0$, (ii) $\mb{E}\{\bar{x}|\mc{F}^y_0\}=\bar{x}$, (iii) $\bar{\tilde{x}}^{k,ex}(t)=0,\, t\geq 0, \, 1 \leq k \leq K$; where (ii) holds since the major agent can compute the real value of $\bar{x}$ by observing its own state. Hence the mean field equation \eqref{MeanFieldEq} reduces to that of completely observed major- minor LQG MFG systems (see \cite{Huang2010}).  \hfill $\square$
 \end{remark}
 \begin{remark}[Estimate of Average Estimation Error: Inf. Pop.]\label{xTildeEst}
The solution to \eqref{estErrVecDynCompact} is given by
\begin{equation}\label{EstErrLimitSol} 
\bar{\tilde{x}}^{ex}(t) = \Phi(t, 0) \bar{\tilde{x}}^{ex}(0)+\int_0^t \Phi(t, \tau)\tilde{\mb{D}}[0_{1\times r}, dw_0^T, 0_{1\times rK}, d\nu_0^T]^Td\tau, \end{equation}
where $\Phi(t, \tau) = \exp\big(\tilde{\mb{A}}(t-\tau)\big)$. The initial estimation error of the minor agent $\mc{A}_i$ is given by
\begin{align}
{\tilde{x}}^{k,ex}_i(0) = - \begin{bmatrix}
 \hat{x}_{i|\mathcal{F}^y_i}(0) -x_i(0)\\
 \hat{x}_{0|\mathcal{F}^y_i}(0) - x_0(0)\\
 \hat{\bar{x}}_{|\mathcal{F}^y_i}(0) - \bar{x}(0)\\
 (\hat{x}_{0|\mathcal{F}^y_0})_{|\mathcal{F}^y_i}(0) - \hat{x}_{0|\mathcal{F}^y_0}(0)\\
( \hat{\bar{x}}_{|\mathcal{F}^y_0})_{|\mathcal{F}^y_i}(0) - \hat{\bar{x}}_{|\mathcal{F}^y_0}(0)
 \end{bmatrix}  = \begin{bmatrix}
 x_i(0)\\
 x_0(0)\\
  0_{nK\times 1}\\
  0_{n \times 1}\\
0_{nK \times 1} \end{bmatrix},
 \end{align}
 since the partial observation information sets $\mc{F}^y_i, 0\leq i \leq N$, at time $t_0=0$ are null sets, the conditional expectations turn into total expectations which according to \textit{Assumption \ref{IntialStateAss}} their value is zero. Hence, the infinite-population limit of the average initial estimation error of the minor agents of type $k$ is given by 
 %\begin{equation}
 $(\bar{\tilde{x}}^{k,ex}(0))^T = [0_{1\times n}, x_0^T(0), 0_{1 \times nK}, 0_{1 \times n}, 0_{1 \times nK}],$
% \end{equation} 
 where $\textit{Assumption \ref{IntialStateAss}}$ is again used, and hence $\mb{E}[\bar{\tilde{x}}^{k,ex}(0)|\mc{F}^y_i]=0$. Then the conditional expectation of $\bar{\tilde{x}}^{ex}(t)$ with respect to $\mc{F}^y_i, 0\leq i \leq N$, i.e. $\hat{\bar{\tilde{x}}}^{ex}_{|\mc{F}^y_i}(t)$, is given by
 \begin{multline}\label{EstErrLimitExp} 
\hat{\bar{\tilde{x}}}^{ex}_{|\mc{F}^y_i}(t):=\mb{E}[\bar{\tilde{x}}^{ex}(t)| \mc{F}^y_i] = \Phi(t, 0) \mb{E}[\bar{\tilde{x}}^{ex}(0)|\mc{F}^y_i]\\+\mb{E}\Big [\int_0^t \Phi(t, \tau)\tilde{\mb{D}}\big[0_{1\times r}, dw_0^T, 0_{1\times rK}, d\nu_0^T\big]^Td\tau\Big|\mc{F}^y_i\Big ] = 0, \end{multline} 
  where the second term is zero due to \textit{Assumption \ref{independence_ass}}. \hfill$\square$
\end{remark} 
\begin{remark}
The setup under consideration yields, in particular, the time invariance of the coefficient $\bar{J}$ of $\bar{\tilde{x}}^{ex}$ in \eqref{MeanFieldEq}, where $\bar{\tilde{x}}^{ex}$ is generated by the dynamics \eqref{estErrVecDynCompact}. However, due to \textit{Observation \ref{xTildeEst}}, $\bar{\tilde{x}}^{ex}$ does not appear in the filter equations \eqref{majorKalman} and \eqref{minorKalman}. 
\hfill$\square$
\end{remark}
%where $\Pi_k$ is computed from \eqref{minorRiccati}. 
%\begin{align}
%&\rho \Pi_0  = \Pi_0 \mathbb{A}_0 + \mathbb{A}_0^T \Pi_0 - \Pi_0 \mathbb{B}_0 R_0^{-1} \mathbb{B}_0^T \Pi_0 + Q_0^\pi, \allowdisplaybreaks\\
%&\rho \Pi_k  = \Pi_k \mathbb{A}_k + \mathbb{A}_k^T \Pi_k - \Pi_k \mathbb{B}_k R^{-1} \mathbb{B}_k^T \Pi_k + Q^{\pi},~~ \forall k,  \allowdisplaybreaks\\  
%&\rho s_{0} ~= \frac{ds_0}{dt} + (\mathbb{A}_0 - \mathbb{B}_0 R_0^{-1} \mathbb{B}_0^T \Pi_0 )^T s_0 + \Pi_0 \mathbb{M}_0 - \bar{\eta}_0, \allowdisplaybreaks\\
%&\rho s_{k} = \frac{ds_k}{dt} + (\mathbb{A}_k - \mathbb{B}_k R^{-1} \mathbb{B}_k^T \Pi_k )^T s_k + \Pi_k \mathbb{M} - \bar{\eta},~~ \forall k, 
%\end{align}
%which we note $\tfrac{ds_0}{dt}=0$ and $\tfrac{ds_k}{dt}=0$ since $\mathbb{M}_0,\, \bar{\eta}_0,\, \mathbb{M},\, \bar{\eta}$ are constant. % in particular, the seventh and eighth equations above are integrated backwards in time from an  infinite horizon.\\
\iffalse
Next we define 
\begin{gather}
M_1 = \begin{bmatrix}
A_1 - B_1 R^{-1} B_1^T \Pi_{1,11} & \text{\Large0}  \\
~~~~~~~~~~~~~~~~~~~~~~~~\ddots &\\
\text{\Large0} & A_K-B_K R^{-1} B_{K}^{T} \Pi_{K,11}  
\end{bmatrix}, \nonumber\\
M_2 = \begin{bmatrix}
B_1R^{-1}B_1^T\Pi_{1,13}\\
\vdots\\
B_K R^{-1} B_K^T \Pi_{K,13} 
\end{bmatrix} ,\,\,\,\,\,\,  M_3 = 
\begin{bmatrix}
A_0 & 0 & 0 \\
\bar{G} & \bar{A} & 0 \\
\bar{G} & -M_2 & M_1
\end{bmatrix}, \nonumber \\
L_{0,H} = Q_0^{1/2} \begin{bmatrix}I, 0, -H_0^\pi \end{bmatrix}.
\end{gather}
\fi
 The final set of assumptions is as follows:
%\begin{ass} \label{ObservabilityAss}
 % The pair $(L_{0,H}, M_3)$ is observable.
%\end{ass}
\begin{assumption} \label{DetectabilityStabilizabilityAss}
The pair $(L_{a}, \mathbb{A}_0 - (\rho /2)I_{n+nK})$ is detectable, and for each $k, 1 \leq k \leq K$, the pair $(L_{b}, \mathbb{A}_k - (\rho /2)I_{3n+2nK})$ is detectable, where $L_a = Q_0^{1/2}[I_n, -H_0^\pi]$ and $L_b = Q^{1/2}[I_n, -H_1, -H_2^\pi, 0_{n \times (n+nK)}]$. The pair $(\mathbb{A}_0-(\rho/2)I_{n+nK}, \mathbb{B}_0)$ is stabilizable and $(\mathbb{A}_k-(\rho/2)I_{3n+2nK}, \mathbb{B}_k)$ is stabilizable for each $k, 1 \leq k \leq K$.
\end{assumption}
\begin{assumption} \label{MFEquationSolAss}
The parameters in \eqref{MajorAgentEq}-\eqref{MinorCostLrgPop} belong to a non-empty set which yields the existence and uniqueness of the solutions ($\Pi_0$, $s_0$, $\Pi_k$, $s_k$, $\bar{A}_k$, $\bar{G}_k$, $\bar{H}_k$, $\bar{L}_k$, $\bar{J}_k$, $\bar{m}_k$) to the resulting set of mean-field fixed-point equations consisting of (\ref{mfeqConsistency}), \eqref{majorRiccati}-\eqref{majorOffset}, and \eqref{minorRiccati}-\eqref{minorOffset}, for which %\begin{align}
%&\mb{A}_0 - \mb{B}_0 R_0^{-1}\mb{B}_0^\intercal \Pi_0 - \tfrac{\rho}{2}I,\\
%&\mb{A}_k - \mb{B}_k R_k^{-1}\mb{B}_k^\intercal \Pi_k - \tfrac{\rho}{2}I, \quad 1 \leq k \leq K,
%\end{align}
%are asymptotically stable, and 
\begin{equation}
\sup_{t\geq 0,1 \leq  k \leq K} e^{-\tfrac{\rho}{2}t}\left( |s_0(t) | + |s_k(t)| + |\bar{m}_k(t)| \right) < \infty.
\end{equation}  
 \end{assumption}
\begin{theorem}[$\epsilon$-Nash Equilibria for PO LQG MM MFG Systems] \label{Thm: POLQGMM-MFG}
Subject to \textit{Assumption \ref{IntialStateAss}- Assumption \ref{MFEquationSolAss}}, the KF-MFG state estimation scheme \eqref{majorKalman}-\eqref{RiccatiEqMajor} and \eqref{RiccatiEqMinor}-\eqref{minorKalman}  together with the Consistency Condition equations \eqref{mfeqConsistency}, \eqref{majorRiccati}-\eqref{majorOffset}, \eqref{minorRiccati}-\eqref{minorOffset} generate an infinite family of stochastic control laws  $\hat{\mathcal{U}}_{MF}^{\infty}$, with finite sub-families $\hat{\mathcal{U}}_{MF}^{N} ~:=~ \{u_i^{\circ};\,  0 \leq i < N \}$, $1 \leq N < \infty $, given by  \eqref{UOptMajEst} and \eqref{minorCntrl}, such that 
%\begin{gather}
%\hat{u}_0^\circ = -R_0^{-1} \mathbb{B}_0^T \left[\Pi_0(\hat{x}_{0|\mathcal{F}_0^y}^T,\hat{\bar{x}}_{|\mathcal{F}_0^y}^T)^T+s_0 \right], \label{majorEstBR}\\
%\hat{u}_i^\circ = -R^{-1} \mathbb{B}_k^T \left [\Pi_k(\hat{x}_{i|\mathcal{F}_i^y}^T,\hat{x}_{0|\mathcal{F}_i^y}^T,\hat{\bar{x}}_{|\mathcal{F}_i^y}^T, (\hat{x}_{0|\mathcal{F}^y_0})_{|\mathcal{F}^y_i}^T,
% (\hat{\bar{x}}_{|\mathcal{F}^y_0})_{|\mathcal{F}^y_i}^T)^T+s_k \right], \label{minorEstBR} 
%\end{gather}
\begin{enumerate}
\item[(i)] $ {\hat{\mathcal{U}}}_{MF}^{\infty}$ yields a unique Nash equilibrium within the set of linear controls  $\mathcal{U}_{y}^{\infty,L}$ such that 
 \begin{equation}
 J_i^{\infty}(u_{i}^{\circ}, u_{-i}^{\circ}) = \inf_{u_i \in \mathcal{U}^{\infty,L}_{y} }J_i^{\infty} (u_i, u_{-i}^{\circ}) \nonumber;
  \end{equation}
\item[(ii)] All agent systems $0 \leq i \leq N$, are $e^{-\frac{\rho}{2}t}$ discounted second order stable in the sense that  for $C$ independent of $N$
\begin{equation*}
\hspace{0mm}\sup_{t\geq 0,\, 0\leq i \leq N} \hspace{-3mm}e^{-\frac{\rho}{2}t} \mathbb{E} \Big ({\Vert \hat{x}_{i|\mathcal{F}^y_i} \Vert}^2+{\Vert \hat{x}_{0|\mathcal{F}^y_i} \Vert}^2 + {\Vert \hat{\bar{x}}_{|\mathcal{F}^y_i}\Vert}^2 +\Vert(\hat{x}_{0|\mathcal{F}^y_0})_{|\mathcal{F}^y_i}\Vert^2+ \Vert(\hat{\bar{x}}_{|\mathcal{F}^y_0})_{|\mathcal{F}^y_i}\Vert^2\Big) < C;
\end{equation*}
 \item[(iii)]  $\{ \hat{\mathcal{U}}_{MF}^N;1\leq N < \infty\}$ yields a unique $\epsilon$-Nash equilibrium within the class of linear control laws $\mathcal{U}_{y}^{N,L}$ for all $\epsilon$, i.e. for all $\epsilon>0$, there exists $N(\epsilon)$ such that for all $N \geq N(\epsilon)$;
\begin{equation*}
J_i^{s,N}(\hat u_i^\circ, \hat u_{-i}^\circ)-\epsilon \leq\inf_{u_i \in\mathcal{U}_{y}^{N,L} } J_i^{s,N}(u_i, \hat u_{-i}^\circ) \leq  J_i^{s,N}(\hat u_i^\circ, \hat u_{-i}^\circ),
\end{equation*}
 where the major agent's and the generic minor agent's performance function $J_i^{s,N}( u_i^\circ, u_{-i}^\circ)$, $0 \leq i \leq N$, is given by
\begin{equation*}
J_i^N(u_i, u_{-i}) + \hat{E}_N,
\end{equation*}
where $J_i^N(u_i, u_{-i})$ is as in the completely observed case, $\hat{E}_N > 0$, and when $u_i = \hat{u}_i^{\circ}$ the following limits hold: 
\begin{itemize}
\item $\lim_{N \rightarrow \infty} J_i^N(\hat u^{\circ}_{i}, \hat u^{\circ}_{-i}) = J_i^{\infty}(\hat{u}_i^{\circ}, \hat{u}^{\circ}_{-i})$,
\item $\lim_{N \rightarrow \infty} \hat{E}_N = \int_{0}^{\infty} e^{- \rho t} \mbox{tr}[Q^{\pi}V]dt$,\\
where $V(t)$ is the solution to (\ref{RiccatiEqMajor}) for the major agent and the solution to (\ref{RiccatiEqMinor}) for a generic minor agent.
\end{itemize}
\end{enumerate}
\end{theorem}
\begin{proof}
% obsolete
%We now apply the standard Separation Theorem method $x_i = \hat x_i + (x_i - \hat x_i) $  and the LQG Major-Minor agent LQG MF method (???) applied to the controlled estimated state equations.
%
%The major agent and individual minor agent state estimation recursive equations schemes are given by the MM KF-MF Equations  above (for size $N$ finite populations and infinite populations).
%
%One next applies the Separation Theorem strategy for reducing a partially observed SOC problem to a completely observed SOC problem for the controlled state estimate processes beginning with the re-expression of the perfomance functions in terms of the state estimation proceses.
%   
%
%The control law dependent summand of the {individual cost} for the {major} agent $A_0$:
%\begin{multline*} J_0^N(u_0, u_{-0})  = \E\int_0^\infty e^{-\rho t}\Big\{\big \lVert  x_0 - {H_0 \hat x^{N}_{|\F_0}-\eta_0}\big\rVert_{Q_0}^2 \\+ \lVert u_0 \rVert_{R_0}^2 \Big\}dt,
%\end{multline*}
%where $\hat x^N_{|\F_0}=(1/N)\sum _{i=1}^N \hat x_{i|\F_0}$.
%
%The control law dependent summand of the {individual cost} for a {minor} agent $A_i,\,\rn$:
%\begin{multline*}
%J_i^N(u_i, u_{-i}) = \E\int_0^\infty e^{-\rho t} \Big\{ \big \lVert \hat x_{i|\F_i^y} - H_1 \hat x_{0|\F_i^y} \\ -  H_2 \hat x^{N}_{|\F_i^y}-\eta \big \rVert_Q^2 + \lVert u_i \rVert_R^2 \Big\}dt.
%\end{multline*}
Generalizing the standard methodology in \cite{CainesBook1988} and \cite{DavisBook1977}, we first decompose the state processes into their estimates and their  estimation errors orthogonal to the corresponding estimates. Substituting the decomposed states into the performance functions and applying the smoothing property of conditional expectations with respect to the increasing filtration families $\mathcal{F}^y_i$ and $\mathcal{F}^y_0$ to the major and minor cost functionals respectively, we obtain the separated performance functions. This technique is applied to both finite and infinite population cases which yields the best response controls $\{\hat{u}_{i}^\circ,~ 0 \leq i \leq N\}$ as optimal tracking controls for the major and minor agents in the infinite population case  (see \cite{KizilkaleTAC2016} for the case where only the minor agent has partial observations on the major agent's state).  Specifically we form the following decompositions where the superscript 's' on the resulting performance functions indicates the separation into control dependent and control independent summands. 
\begin{enumerate}
\item Major Agent's State Decomposition\\
\textit{Finite Population:}
\begin{equation*}
 \left [ \begin{array}{c} x_0 \\ x^{(N)} \end{array}  \right ] = \left [ \begin{array}{c} \hat{x}_{0|\mathcal{F}^y_0}  \\  \hat{x}^{(N)}_{|\mathcal{F}^{y}_{0}} \end{array}\right ] + \left [ \begin{array}{c} x_0-\hat{x}_{0|\mathcal{F}^y_0}  \\ x^{(N)} - \hat{x}^{(N)}_{|\mathcal{F}^{y}_{0}} \end{array}\right] .
 \end{equation*}
\textit{Infinite Population:}
\[ \left [ \begin{array}{c} x_0 \\ \bar x \end{array}  \right ] = \left [ \begin{array}{c} \hat{x}_{0|\mathcal{F}^y_0}  \\ \hat{\bar{ x}}_{|\mathcal{F}^{y}_{0}} \end{array}\right ] + \left [ \begin{array}{c} x_0-\hat{x}_{0|\mathcal{F}^y_0}  \\ \bar x - \hat{\bar{x}}_{|\mathcal{F}^{y}_{0}} \end{array}\right] . \] 

\item Major Agent's Cost Functional Separation\\
\textit{Finite Population:}
\begin{multline}\label{eq:sepPF1}
\hspace{-5mm} J_0^{s, N}(u_0, u_{-0}) = \mathbb{E} \bigg [ \int_0^\infty e^{-\rho t}  \Big \{ \big \lVert  \hat{x}_{0|\mathcal{F}^y_0}  - H_0 \hat{x}^{(N)}_{|\mathcal{F}^y_0}  -\eta_0 \big\rVert_{Q_0}^2 
+ \lVert u_0 \rVert_{R_0}^2 \Big\}dt \bigg ] \allowdisplaybreaks\\ + \mathbb{E} \bigg [ \int_0^\infty e^{-\rho t} \big \lVert (x_0-\hat{x}_{0|\mathcal{F}^y_0})  - 
H_0 ( x^{(N)} - \hat {x}^{(N)}_{|\mathcal{F}^y_0}) \big \rVert_{Q_0}^2 dt \bigg].
 \end{multline}
 \textit{Infinite Population:}
\begin{multline}\label{eq:sepPF2}
\hspace{-5mm} J_0^{s, \infty} = \mathbb{E} \bigg [ \int_0^\infty e^{-\rho t} \Big \{ \big \lVert  \hat{x}_{0|\mathcal{F}^y_0}  - H_0^\pi \hat{\bar x}_{|\mathcal{F}^y_0}  - \eta_0 \big\rVert_{Q_0}^2
  + \lVert u_0 \rVert_{R_0}^2 \Big\}dt \bigg ] \allowdisplaybreaks\\+ \mathbb{E} \bigg [ \int_0^\infty e^{-\rho t} \big \lVert (x_0-\hat{x}_{0|\mathcal{F}^y_0}  )- 
H_0^\pi (\bar x - \hat {\bar x}_{|\mathcal{F}^y_0}) \big \rVert_{Q_0}^2 dt \bigg].
 \end{multline}

\item Minor Agent's State Decomposition\\
\textit{Finite Population:}
\[ \left [ \begin{array}{c}x_i\\x_0\\ x^{(N)} \end{array} \right ] = \left [ \begin{array}{c} \hat x_{i|\mathcal{F}_i^y} \\ \hat x_{0|\mathcal{F}_i^y}\\ {\hat{x}^{(N)}}_{|\mathcal{F}_i^y} \end{array}\right ] + \left [ \begin{array}{c} x_i - \hat x_{i|\mathcal{F}_i^y} \\ x_0 - \hat x_{0|\mathcal{F}_i^y}\\ x^{(N)} -\hat{x}^{(N)}_{|\mathcal{F}_i^y}\end{array}\right] . \] 
\textit{Infinite Population:}
\[ \left [ \begin{array}{c}x_i\\x_0\\ \bar x \end{array} \right ] = \left [ \begin{array}{c} \hat x_{i|\mathcal{F}_i^y} \\ \hat x_{0|\mathcal{F}_i^y}\\ {\hat{\bar x}}_{|\mathcal{F}_i^y} \end{array}\right ] + \left [ \begin{array}{c} x_i - \hat x_{i|\mathcal{F}_i^y} \\ x_0 - \hat x_{0|\mathcal{F}_i^y}\\ \bar x -\hat{ {\bar x}}_{|\mathcal{F}_i^y}\end{array}\right] . \] 
\item Minor Agent's Cost Functional Separation\\
\textit{Finite Population:}
\begin{multline}\label{eq:sepPF3}
J_i^{s, N}(u_i, u_{-i}) = \mathbb{E} \bigg [ \int_0^\infty e^{-\rho t} \Big\{ \big \lVert \hat x_{i|\mathcal{F}_i^y} - H_1 \hat x_{0|\mathcal{F}_i^y} \allowdisplaybreaks\\ 
-  H_2 \hat{x}^{(N)}_{|\mathcal{F}_i^y}-\eta \big \rVert_Q^2 + \lVert u_i \rVert_R^2 \Big\}dt \bigg] 
+ \mathbb{E} \bigg [ \int_0^\infty e^{-\rho t} \big \lVert (x_i-\hat x_{i|\mathcal{F}_i^y})  \allowdisplaybreaks\\ - H_1 (x_0-\hat x_{0|\mathcal{F}_i^y}) - H_2(x^{(N)} - \hat{x}^{(N)}_{|\mathcal{F}_i^y}) \big \rVert^2_Q  dt\bigg] .
\end{multline}
\newpage
\textit{Infinite Population:}
\begin{multline}\label{eq:sepPF4}
J_i^{s, \infty} = \mathbb{E} \bigg [ \int_0^\infty e^{-\rho t} \Big\{ \big \lVert \hat x_{i|\mathcal{F}_i^y} - H_1 \hat x_{0|\mathcal{F}_i^y} \allowdisplaybreaks\\ 
-  H_2^\pi \hat{\bar x}_{|\mathcal{F}_i^y}-\eta \big \rVert_Q^2 + \lVert u_i \rVert_R^2 \Big\}dt \bigg] 
+ \mathbb{E} \bigg [ \int_0^\infty e^{-\rho t} \big \lVert (x_i-\hat x_{i|\mathcal{F}_i^y}) \allowdisplaybreaks\\ - H_1 (x_0-\hat x_{0|\mathcal{F}_i^y}) - H_2^\pi(\bar x - \hat{\bar x}_{|\mathcal{F}_i^y}) \big \rVert^2_Q  dt\bigg] .
\end{multline}
\end{enumerate}
As can be seen, the first integral expressions in \eqref{eq:sepPF1}, \eqref{eq:sepPF2}, \eqref{eq:sepPF3} and \eqref{eq:sepPF4}  depend on the estimated states generated by the estimation schemes \eqref{majorKalman} and \eqref{minorKalman} for the major agent and minor agents respectively, and the second integral expressions depend only upon the respective estimation errors and on the solutions to the associated Riccati equations. The latter expressions are independent of the control actions and generate the additional cost $\hat{E}_N$ in the finite population case incurred by the errors in the estimation process. 
%Next, the resulting infinite population tracking problems are solved %for the major and minor agents in their separated forms with the %controlled state in 44 replacing that in 51. This is achieved by use %of the  MM LQG MFG  equations   blah blah.

Next, the resulting infinite population tracking problems are solved for the major and minor agents in their separated forms. The control dependent summands in \eqref{eq:sepPF2} have exactly the same structure in terms of the functional dependence on the estimated states as the infinite population cost functionals in the complete observation case have on the states. Moreover, the control dependent summands in \eqref{eq:sepPF4} have exactly the same structure in terms of the functional dependence on the estimated states as the infinite population cost functional for the system \eqref{minorSys_noUhat} with complete observations on its own state, the major agent's state, and the major agent's estimates of its own state and the mean field. Hence, by the Separation Principle the infinite population Nash Certainly Equivalence equilibrium controls are given by $\{\hat{u}_{i}^\circ,~ 0 \leq i \leq N\}$ in the theorem statement. Finally the infinite population control actions are applied to the finite population systems and the fact that these yield (i) $e^{-\frac{\rho}{2}t}$ second order system stability, and (ii) $\epsilon$-Nash equilibrium property, is established by the standard approximation analysis parallel to  that of completely observed major-minor LQG MFG systems (see \cite{HuangTAC2007}, \cite{Huang2010}).      
  \end{proof}

\begin{remark} We note that $(\hat{x}_{0|\mathcal{F}^y_0})_{|\mathcal{F}^y_i}$ and $(\hat{\bar{x}}_{|\mathcal{F}^y_0})_{|\mathcal{F}^y_i}$ do not appear in the minor agent's state decomposition and in its separated performance function but that they are used in the extended estimated state recursion \eqref{minorKalman} and hence appear in the control action for a minor agent in \eqref{minorCntrl}.
\end{remark}
\begin{remark} The non-uniqueness of Nash equilibria which may occur in classical LQG stochastic dynamic games with specified information sets \cite{Basar1976,Basar1977} does not occur in this analysis. This holds since, for the specified maximal individual information sets, and subject to the hypotheses of \textit{\Cref{Thm: POLQGMM-MFG}} giving unique solutions to the MFG Consistency equations (as functions of the system parameters), a unique linear best response function is obtained for each  agent with respect to its stochastic control problem arising from its performance function in the infinite population limit. We note that any set of controls generating a Nash equilibrium will yield the same consistency equations whose solution depends only on the system parameters.
\end{remark}

\section{Simulations} \label{sec:simulation}
Consider a system of 100 minor agents and a single major agent. The system matrices $\{A_k,~ B_k,~ 1 \leq k \leq 100\}$ for the minor agents are uniformly defined as
\begin{equation*}
 {A} := \left[ \begin{array}{cc}
       -0.05 & -2 \\
        1 & 0  \end{array} \right],~~  B := \left[ \begin{array}{cc}
        1 \\
        0 \end{array} \right], 
\end{equation*}
and for the major agent we have
\begin{equation*}
 A_0 := \left[ \begin{array}{cc}
       -1 & -1 \\
       1 & 0  \end{array} \right], \quad B_0 := \left[ \begin{array}{cc}
       1 \\
        0 \end{array} \right].
\end{equation*}
The parameters used in the simulation are: $ t_{final} = 25 \,\text{sec},~\Delta t = 0.01\, \text{sec},~\sigma_{w_0} = \sigma_{w_i} = 0.009,~ \sigma_{v_0} = \sigma_{v_i} = 0.0003,~ \rho = 0.9,~ \eta_0 = \eta = [0.25,0.25]^T,~ Q_0 = Q = I_{2\times 2},\, R_0 = R = 1,~ H_0 = H_1 = H_2 = 0.6\times I_{2\times 2},~ G = 0_{2\times 2}$. The true and estimated state trajectories, and the estimation errors for a single realization can be displayed for the entire population of 101 agents together, but in figures \ref{fig:MajorTrajAndEst}-\ref{fig:MinorEstError} only 10 minor agents are shown for the sake of clarity.% Moreover, the steady state values for the estimation error covariance matrices of the major agent, $V_0^{\infty}$, and minor agents, $V^{\infty}$, are given by

\begin{figure}
\centering
\includegraphics[width=0.75\linewidth]{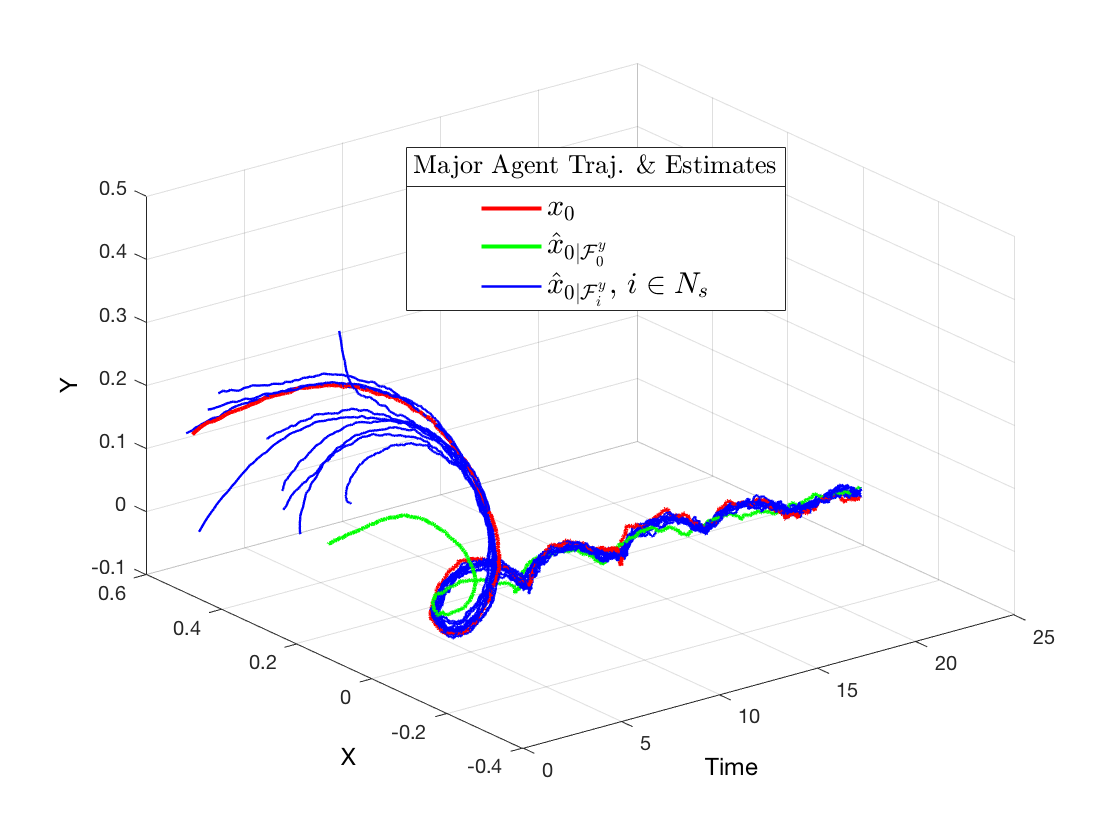}
\caption{The Major agent's true and estimated trajectories.}
\label{fig:MajorTrajAndEst}
\end{figure}

\begin{figure}
\centering
\includegraphics[width=0.75\linewidth]{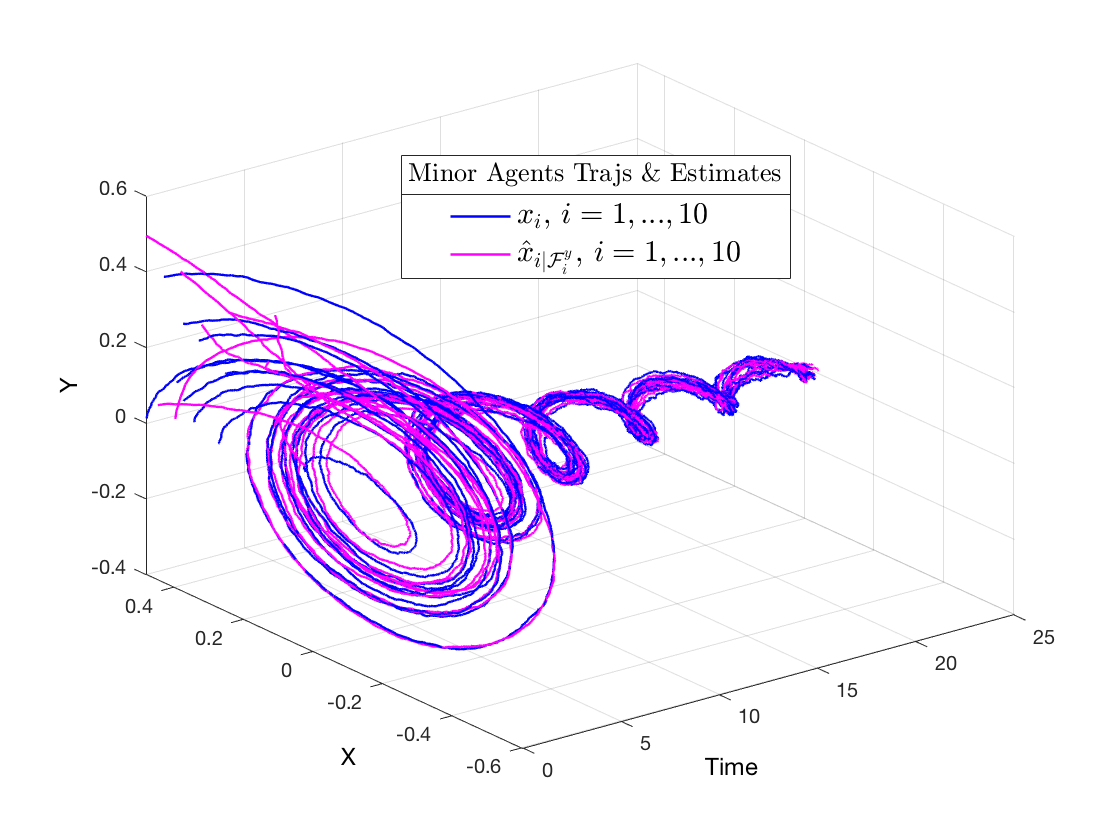}
\caption{10 Minor agents' true and estimated trajectories.} \label{fig:AllPO}
\end{figure}

\begin{figure}
\centering
\includegraphics[width=0.75\linewidth]{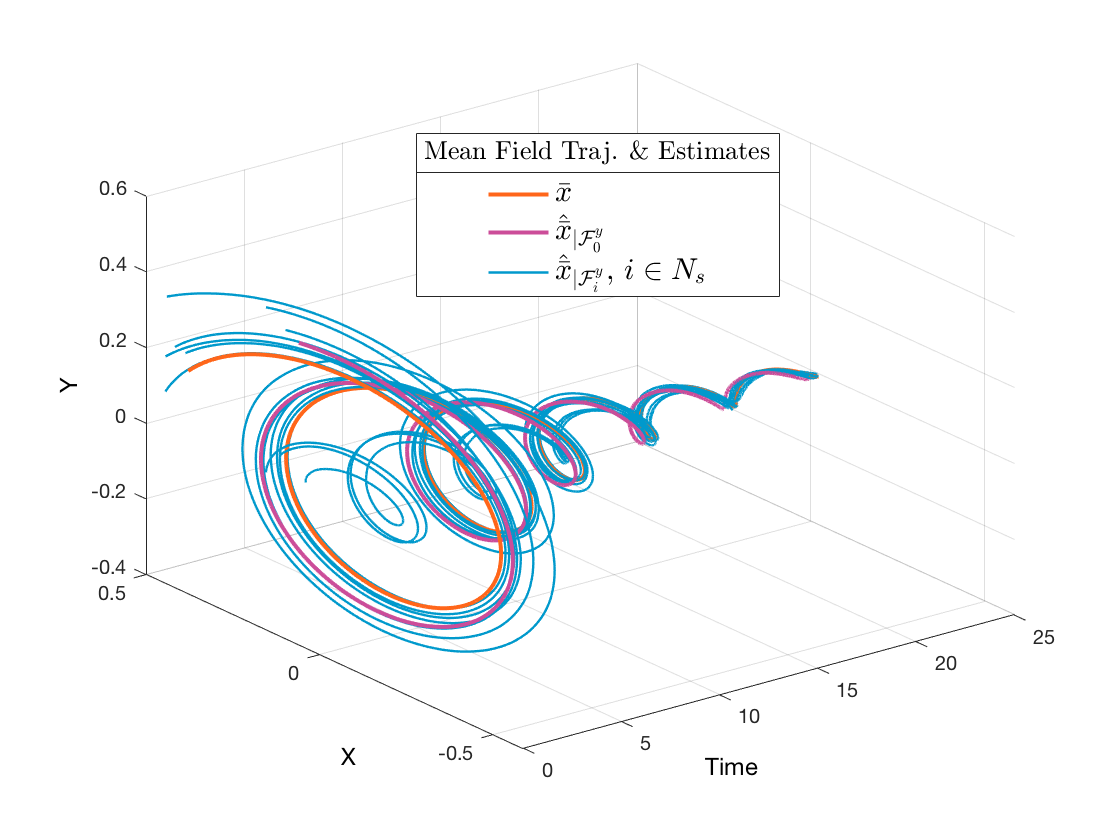}
\caption{The mean field true and estimated trajectories.} \label{fig:MFTrajAndEst}
\end{figure}

\begin{figure}
\centering
\includegraphics[width=0.75\linewidth]{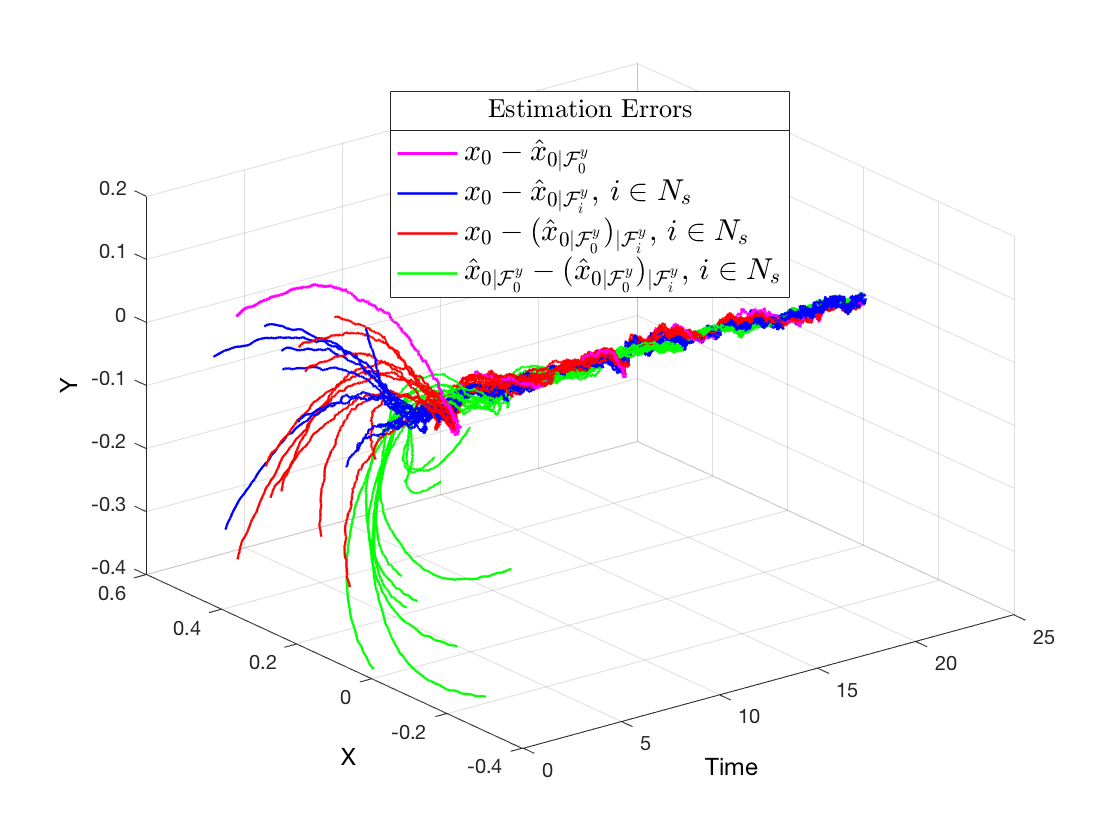}
\caption{The estimation errors of the major agent's trajectory.} \label{fig:MajorEstError}
\end{figure}

\begin{figure}
\centering
\includegraphics[width=0.75\linewidth]{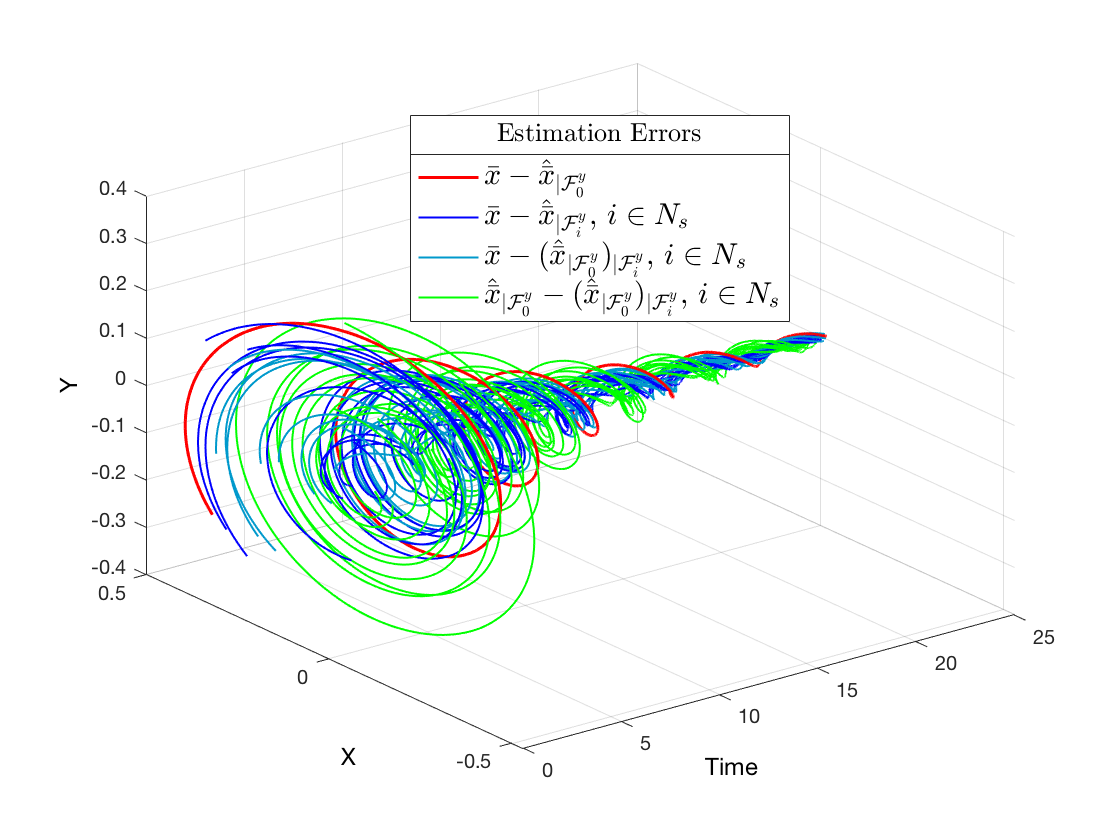}
\caption{The estimation errors of the mean field trajectory.} \label{fig:MFEstError}
\end{figure}

\begin{figure}
\centering
\includegraphics[width=0.75\linewidth]{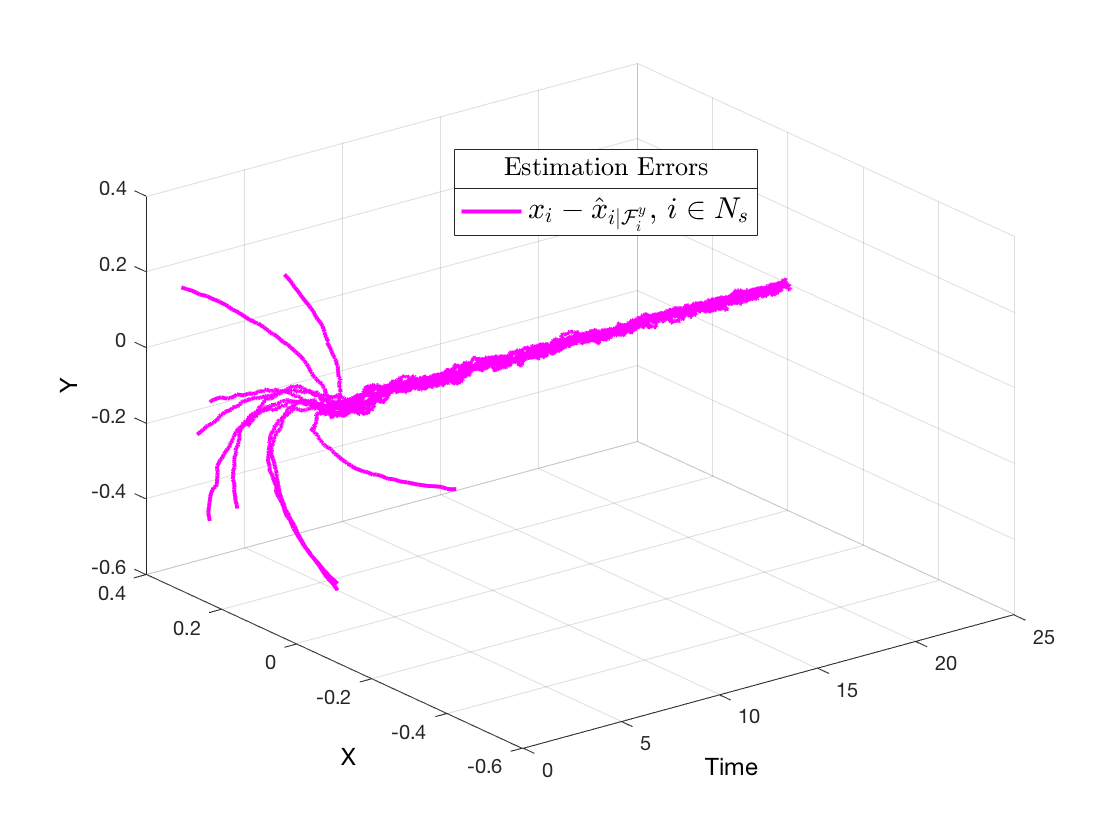}
\caption{The estimation errors of 10 minor agents' trajectories.} \label{fig:MinorEstError}
\end{figure}

\section{Conclusions}\label{sec:conclusions}
In this paper, PO MM LQG MFG problems with general information patterns are studied where (i) the major agent has partial observations on its own state, and (ii) each minor agent has partial observations on its own state and the major agent's state. For a general case LQG MFG systems, the existence of $\epsilon$-Nash equilibria together with the individual agents' control laws generating them are established via the Separation Principle. The assumption of partial observations for all agents leads to a new situation involving the recursive estimation by each minor agent of the major agent's estimate of its own state. To the best of our knowledge, the dynamic game theoretic equilibrium which is established in this paper constitutes a rare case wherein agents explicitly generate estimates of another agent's beliefs. Moreover,  this does not give rise to an infinite regress due to the information asymmetry of the major and minor agents.  
%\vspace{-4mm}

\appendix
\section{Time-Invariance of Finite-Population State Feedback Controls}\label{appendix}

If we write down the extended linear dynamics and quadratic cost functionals for the agents in the finite-population case, the system matrices and the control matrix coefficients in the dynamics (or filter equations) are time invariant. Moreover, the weight matrices in the cost functionals are time-invariant, and hence we have time-invariant infinite-horizon cost functionals. The optimal state linear feedback controls for classical time-invariant infinite-horizon linear quadratic problems would have time-invariant coefficients [33]. Therefore when we want to obtain an ansatz for the mean field equation, we consider the most general time-invariant state feedback controls for each agent. To make this more clear, we write down the finite-population extended systems for minor agents for both the complete and partial observations case. For the purpose of illustration we assume that all minor agents are of the same type.\\

%\begin{itemize}
\subsection{Complete Observations Case} \label{sec:CO}
The extended finite-population system for the major agent and a generic minor agent is given by stacking dynamics \eqref{MajorAgentEq} and \eqref{MinorAgentEq} as follows.
\begin{multline}
\begin{bmatrix}
dx_0\\
dx_1\\
\vdots\\
dx_i\\
\vdots\\
dx_N
\end{bmatrix} = \begin{bmatrix}
A_0 & 0 & 0 & \dots & 0 & \dots & 0\\
G & A & 0 & \dots & 0 & \dots & 0\\
\vdots& & &\vdots & &  &\vdots \\
G & 0 & 0 & \dots & A & \dots & 0\\
\vdots& & &\vdots & &  &\vdots \\
G & 0 & 0 & \dots & 0 & \dots & A
\end{bmatrix} \begin{bmatrix}
x_0\\
x_1\\
\vdots\\
x_i\\
\vdots\\
x_N
\end{bmatrix} dt + \begin{bmatrix}
B_0 & 0 & 0 & \dots & 0 & \dots & 0\\
0 & B & 0 & \dots & 0 & \dots & 0\\
\vdots& & &\vdots & &  &\vdots \\
0 & 0 & 0 & \dots & B & \dots & 0\\
\vdots& & &\vdots & &  &\vdots \\
0 & 0 & 0 & \dots & 0 & \dots & B
\end{bmatrix} \\ \times \begin{bmatrix}
u_0\\
u_1\\
\vdots\\
u_i\\
\vdots\\
u_N
\end{bmatrix} dt + \begin{bmatrix}
D_0 & 0 & 0 & \dots & 0 & \dots & 0\\
0 & D & 0 & \dots & 0 & \dots & 0\\
\vdots& & &\vdots & &  &\vdots \\
0 & 0 & 0 & \dots & D & \dots & 0\\
\vdots& & &\vdots & &  &\vdots \\
0 & 0 & 0 & \dots & 0 & \dots & D
\end{bmatrix} \begin{bmatrix}
dw_0\\
dw_1\\
\vdots\\
dw_i\\
\vdots\\
dw_N
\end{bmatrix}. \label{app:finDynCompObs}
\end{multline}
 Equivalently the above system can be written as 
\begin{equation}
dx^{ex} = (\tf{A} x^{ex} + \sum_{i=0}^{N}\tf{B}_iu_i)dt + \tf{D} d\tf{W},\label{app:dynMajorCompleteObs} 
\end{equation}
where $\{\tf{B}_i,\, 0\leq i\leq N\}$ correspond to the $N+1$ columns of the control matrix coefficient in \eqref{app:finDynCompObs}. 
The cost functional for each agent in terms of the extended state $x^{ex}$ can be written as
\begin{equation} 
J_{i}^{N,ex}(u_{i}, u_{-i}) = \mathbb{E} \int_{0}^{\infty}  e^{-\rho t} \Big\{ \Vert {x}^{ex}\Vert _{\tf{Q}_i}^2 -2(x^{ex})^T \gamma_i+ \Vert u_i \Vert_{R}^2 \Big\}dt, \quad 0\leq i \leq N. \label{app:costMajorCompleteObs}
\end{equation}
Given that $\tf{A}$ and $\tf{B}_i, 1\leq i \leq N$, in the linear dynamics \eqref{app:dynMajorCompleteObs} are time-invariant, and the cost functional \eqref{app:costMajorCompleteObs} is quadratic, infinite-horizon and time-invariant,
the Nash-equilibrium strategies $u_i, 0\leq i \leq N$, are given by 
\begin{equation}
u_i =  - R^{-1} \tf{B}_i^T ( \Pi_i x^{ex} + s_i), \quad 0\leq i \leq N,
\end{equation}
subject to the set of coupled Riccati equations 
\begin{equation}
\rho \Pi_i = F_i^T \Pi_i + \Pi_i F_i - \Pi_i \textsf{B}_i R^{-1}\textsf{B}_i^T\Pi_i+ \tf{Q}_i,
\end{equation}
and the set of coupled offset equations 
\begin{gather}
\rho s_{i} = F_i^T s_i + \Pi_i \textsf{M}_ - {\gamma}_i,\label{app:Offset}
\end{gather}
where $F_i = \tf{A}-\sum_{j=0,j\neq i}^{N}\tf{B}_j R^{-1} \tf{B}_j^T \Pi_j $.

Therefore, when we want to obtain an ansatz for the mean field equation for the completely observed major minor LQG mean field game (MM LQG MFG) systems, we consider the most general time-invariant state feedback controls for a generic minor agent as 
\begin{equation}\label{app:cntrlgenCompObs}
u_i = L_1 x_i + L_2 x_0 + L_3 \sum_{j=1}^{N} x_j + m.
\end{equation}  
Then we substitute \eqref{app:cntrlgenCompObs} in the dynamics of each agent, take the average over population and then its limit as $N \rightarrow \infty$, to obtain the ansatz for the mean field equation as 
\begin{equation}
d\bar{x} = (\bar{A} \bar{x} + \bar{G} x_0 + \bar{m})dt,   \quad  \text{q.m.}
\end{equation}

\subsection{Partial Observations Case}

We note that for the information patterns considered in this paper, the problem of partially observed MM LQG MFG systems in the finite population case is an open problem. This is because in order to compute its best response strategy (i.e. feedback control actions yielding a Nash  equilibrium), an agent requires estimates of other agents' states, and for that they need to estimate other agents' strategies; this leads to the generation of second-order and higher-order estimates. Hence each agent will have an infinite-dimensional extended system. We remark that still in this case the system matrix and the control matrix coefficient would be time-invariant. Moreover, the weight matrices in the infinite-horizon extended cost functional are time-invariant. \\
The alternative is to solve the infinite-population limit of the problem, where the major agent is not impacted by the individual minor agents but each minor agent is impacted by the major agent. Hence, in the infinite-population limit, to compute their best response strategies (i.e. those yielding a Nash equilibrium) the major agent does not estimate the minor agents' states but the mean field. This is while each minor agent estimates the major agent's state and the mean field, for which it estimates the major agent's strategy leading to the generation of the second-order estimates. Therefore, in the infinite population case the major agent only generates first-order estimates, and each minor agent generates first-order and second-order estimates. Using this fact to obtain an ansatz for the mean field equation, only first-order estimates generated by the major agent, and first-order and second-order estimates generated by minor agents play a role. Accordingly, in the following we give the filter equations for the major agent and minor agents in the finite-population case, and subsequently show that in the finite population case time-invariant linear state feedback controls are optimal (in the sense of yielding a Nash equilibrium) for the agents' extended systems.
\subsubsection{Major Agent:}
The filter equation for the major agent in the finite population case is given by
\begin{multline}
\begin{bmatrix}
d\hat{x}_{0|\mc{F}^y_0}\\
d\hat{x}_{1|\mc{F}^y_0}\\
\vdots\\
d\hat{x}_{i|\mc{F}^y_0}\\
\vdots\\
d\hat{x}_{N|\mc{F}^y_0}
\end{bmatrix} = \begin{bmatrix}
A_0 & 0 & 0 & \dots & 0 & \dots & 0\\
G & A & 0 & \dots & 0 & \dots & 0\\
\vdots& & &\vdots & &  &\vdots \\
G & 0 & 0 & \dots & A & \dots & 0\\
\vdots& & &\vdots & &  &\vdots \\
G & 0 & 0 & \dots & 0 & \dots & A
\end{bmatrix} \begin{bmatrix}
\hat{x}_{0|\mc{F}^y_0}\\
\hat{x}_{1|\mc{F}^y_0}\\
\vdots\\
\hat{x}_{i|\mc{F}^y_0}\\
\vdots\\
\hat{x}_{N|\mc{F}^y_0}
\end{bmatrix} dt \\+ \begin{bmatrix}
B_0 & 0 & 0 & \dots & 0 & \dots & 0\\
0 & B & 0 & \dots & 0 & \dots & 0\\
\vdots& & &\vdots & &  &\vdots \\
0 & 0 & 0 & \dots & B & \dots & 0\\
\vdots& & &\vdots & &  &\vdots \\
0 & 0 & 0 & \dots & 0 & \dots & B
\end{bmatrix} \begin{bmatrix}
\hat{u}_{0|\mc{F}^y_0}\\
\hat{u}_{1|\mc{F}^y_0}\\
\vdots\\
\hat{u}_{i|\mc{F}^y_0}\\
\vdots\\
\hat{u}_{N|\mc{F}^y_0}
\end{bmatrix} dt + K_0 d\nu_0,  \label{app:finDynPartialObs}
\end{multline}
Equivalently the above system can be written as 
\begin{equation}
d\hat{x}^{ex}_{|\mc{F}^y_0} = (\tf{A} \hat{x}^{ex}_{|\mc{F}^y_0} + \sum_{i=0}^{N}\tf{B}_i\hat{u}_{i|\mc{F}^y_0})dt + K_0 d\nu_0 \label{app:dynMajorPartialObs}
\end{equation}
The cost functional of the major agent in the finite population case is given by
\begin{multline} 
J_{0}^{N, ex}(u_{0}, u_{-0}) = \mathbb{E} \int_{0}^{\infty}  e^{-\rho t} \Big\{ \Vert \hat{x}^{ex}_{|\mc{F}^y_0}\Vert _{\tf{Q}_0}^2 -2(\hat{x}^{ex}_{|\mc{F}^y_0})^T \gamma_0+ \Vert u_0 \Vert_{R_0}^2 \Big \}dt \\
+ \mathbb{E} \int_{0}^{\infty} e^{-\rho t}\Big\{\Vert x^{ex}- \hat{x}^{ex}_{|\mc{F}^y_0}\Vert _{\tf{Q}_0}^2 -2(x^{ex}- \hat{x}^{ex}_{|\mc{F}^y_0})^T \gamma_0\Big\}dt, \label{app:costMajorPartialObs}
\end{multline}
Following the results of Section \ref{sec:CO}, given that $\tf{A}$ and $\tf{B}_i, 1\leq i \leq N$, in the linear dynamics \eqref{app:dynMajorPartialObs} are time-invariant, and the cost functional \eqref{app:costMajorPartialObs} is quadratic, infinite-horizon and time-invariant, we consider the following time-invariant linear state feedback controls as ansatze for  the optimal controls $\hat{u}_{i|\mc{F}^y_0}, 0 \leq i \leq N$,
\begin{gather}
\hat{u}_{0|\mc{F}^y_0} = L_1^0 \hat{x}_{0|\mc{F}^y_0} + L_2^0 \sum_{j=1}^{N}\hat{x}_{j|\mc{F}^y_0} + L_3^0 \label{app:majorMajorAnsatz}\\
\hat{u}_{i|\mc{F}^y_0} = L_4^0 \hat{x}_{i|\mc{F}^y_0} + L_5^0 \hat{x}_{0|\mc{F}^y_0}  + L_6^0\sum_{j=1}^{N}\hat{x}_{j|\mc{F}^y_0} + L_7^0 \label{app:majorMinorAnsatz}
\end{gather}
Note that all coefficients $L_1^0,\, L_2^0,\, L_3^0,\,L_4^0,\, L_5^0, \, L_6^0,\, L_7^0       $, are time-invariant. \\
%\newpage
%\begin{multline}
%\begin{bmatrix}
%d\hat{x}_{0|\mc{F}^y_i}\\
%d\hat{x}_{1|\mc{F}^y_i}\\
%d\hat{x}_{2|\mc{F}^y_i}\\
%\vdots\\
%d\hat{x}_{i|\mc{F}^y_i}\\
%\vdots\\
%d\hat{x}_{N|\mc{F}^y_i}\\
%\end{bmatrix} = \begin{bmatrix}
%A_0 & 0 & 0 & \dots & 0 & \dots & 0\\
%G & A & 0 & \dots & 0 & \dots & 0\\
%G & 0 & A & \dots & 0 & \dots & 0\\
%\vdots& & &\vdots & &  &\vdots \\
%G & 0 & 0 & \dots & A & \dots & 0\\
%\vdots& & &\vdots & &  &\vdots \\
%G & 0 & 0 & \dots & 0 & \dots & A
%\end{bmatrix} \begin{bmatrix}
%\hat{x}_{0|\mc{F}^y_i}\\
%\hat{x}_{1|\mc{F}^y_i}\\
%\hat{x}_{2|\mc{F}^y_i}\\
%\vdots\\
%\hat{x}_{i|\mc{F}^y_i}\\
%\vdots\\
%\hat{x}_{N|\mc{F}^y_i}
%\end{bmatrix} dt \\+ \begin{bmatrix}
%B_0 & 0 & 0 & \dots & 0 & \dots & 0\\
%0 & B & 0 & \dots & 0 & \dots & 0\\
%0 & 0 & B & \dots & 0 & \dots & 0\\
%\vdots& & &\vdots & &  &\vdots \\
%0 & 0 & 0 & \dots & B & \dots & 0\\
%\vdots& & &\vdots & &  &\vdots \\
%0 & 0 & 0 & \dots & 0 & \dots & B
%\end{bmatrix} \begin{bmatrix}
%\hat{u}_0\\
%\hat{u}_1\\
%\hat{u}_2\\
%\vdots\\
%\hat{u}_i\\
%\vdots\\
%\hat{u}_N
%\end{bmatrix} dt + K_i d\nu_i,  \label{finDynPartialObs}
%\end{multline}
\subsubsection{Minor Agent:}

The filter equation for a minor agent $\mc{A}_i,\, 1\leq i \leq N$, in the finite population case is given by
\begin{multline}
\hspace{-1cm}
\begin{bmatrix}
d\hat{x}_{0|\mc{F}^y_i}\\
d\hat{x}_{1|\mc{F}^y_i}\\
\vdots\\
d\hat{x}_{i|\mc{F}^y_i}\\
\vdots\\
d\hat{x}_{N|\mc{F}^y_i}\\
d(\hat{x}_{0|\mc{F}^y_0})_{|\mc{F}^y_i}\\
d(\hat{x}_{1|\mc{F}^y_0})_{|\mc{F}^y_i}\\
\vdots\\
d(\hat{x}_{N|\mc{F}^y_0})_{|\mc{F}^y_i}
\end{bmatrix} = \begin{bmatrix}
A_0 & 0 & \dots & 0 & \dots & 0 & B_0 L^0_1 & B_0 L^0_2 & \dots & B_0 L^0_2\\
G & A & \dots & 0 & \dots & 0 & 0 & 0 & \dots & 0\\
\vdots&  & & &  &\vdots &  & & \vdots & \\
G & 0  & \dots & A & \dots & 0& 0 & 0 & \dots & 0\\
\vdots&  & & &  &\vdots &  &  & \vdots & \\
G & 0  & \dots & 0 & \dots & A & 0 & 0 & \dots & 0\\
0 & 0 & \dots & 0 & \dots & 0 & A_0+B_0 L^0_1 & B_0 L^0_2 & \dots & B_0 L^0_2\\
0 & 0 & \dots & 0 & \dots & 0 & G+BL^0_5 & A +B(L^0_4+L^0_6) &\dots & B L^0_6\\
\vdots&  & & &  &\vdots &  &  & \vdots & \\
0 & 0 & \dots & 0 & \dots & 0 & G+BL^0_5 & B L^0_6 & \dots & A +B(L^0_4+L^0_6)
\end{bmatrix} \\ \times \begin{bmatrix}
\hat{x}_{0|\mc{F}^y_i}\\
\hat{x}_{1|\mc{F}^y_i}\\
\vdots\\
\hat{x}_{i|\mc{F}^y_i}\\
\vdots\\
\hat{x}_{N|\mc{F}^y_i}\\
(\hat{x}_{0|\mc{F}^y_0})_{|\mc{F}^y_i}\\
(\hat{x}_{1|\mc{F}^y_0})_{|\mc{F}^y_i}\\
\vdots\\
(\hat{x}_{N|\mc{F}^y_0})_{|\mc{F}^y_i}
\end{bmatrix} dt + \begin{bmatrix}
 0 &  \dots & 0 & \dots & 0\\
 B &  \dots & 0 & \dots & 0\\
  & \vdots & &  &\vdots \\
 0 &  \dots & B & \dots & 0\\
 & \vdots & &  &\vdots \\
 0  & \dots & 0 & \dots & B\\
 0  & \dots & 0 & \dots & 0\\
 0  & \dots & 0 & \dots & 0\\
 & \vdots & &  &\vdots \\
 0  & \dots & 0 & \dots & 0\\
\end{bmatrix} \begin{bmatrix}
\hat{u}_{1|\mc{F}^y_i}\\
\vdots\\
\hat{u}_{i|\mc{F}^y_i}\\
\vdots\\
\hat{u}_{N|\mc{F}^y_i}
\end{bmatrix} dt + \begin{bmatrix}
L^0_3\\
0\\
\vdots\\
0\\
\vdots\\
0\\
L^0_3\\
L^0_7\\
\vdots\\
L^0_7
\end{bmatrix} dt + K_i d\nu_i, \quad 1\leq i \leq N,
 \label{app:finDynPartialObs}
\end{multline}
where \eqref{app:majorMajorAnsatz} and \eqref{app:majorMinorAnsatz} have been substituted.  
Equivalently, the above system can be written as
\begin{equation} 
d\begin{bmatrix}\hat{x}^{ex}_{|\mc{F}^y_i}\\(\hat{x}^{ex}_{|\mc{F}^y_0})_{|\mc{F}^y_i}\end{bmatrix}= \left(\tbf{A} \begin{bmatrix}\hat{x}^{ex}_{|\mc{F}^y_i}\\(\hat{x}^{ex}_{|\mc{F}^y_0})_{|\mc{F}^y_i}\end{bmatrix} + \sum_{j=0}^{N}\tbf{B}_j\hat{u}_{j|\mc{F}^y_i} + \tbf{M}\right)dt + K_i d\nu_i. \label{app:dynMinorPartialObs}
\end{equation}
The cost functional of the minor agent in the finite population case is given by
\begin{multline} 
J_{i}^{N, ex}(u_{i}, u_{-i}) = \mathbb{E} \int_{0}^{\infty}  e^{-\rho t} \left\{ \bigg\Vert \begin{bmatrix}\hat{x}^{ex}_{|\mc{F}^y_i}\\(\hat{x}^{ex}_{|\mc{F}^y_0})_{|\mc{F}^y_i}\end{bmatrix}\bigg\Vert _{\textbf{\tf{Q}}_i}^2 -2\bigg(\begin{bmatrix}\hat{x}^{ex}_{|\mc{F}^y_i}\\(\hat{x}^{ex}_{|\mc{F}^y_0})_{|\mc{F}^y_i}\end{bmatrix}\bigg)^T{\Gamma}_i+ \Vert u_i \Vert_{R}^2 \right\}dt\\
\mathbb{E} \int_{0}^{\infty}  e^{-\rho t} \left\{ \bigg\Vert \begin{bmatrix}x^{ex}-\hat{x}^{ex}_{|\mc{F}^y_i}\\\hat{x}^{ex}_{|\mc{F}^y_0}-(\hat{x}^{ex}_{|\mc{F}^y_0})_{|\mc{F}^y_i}\end{bmatrix}\bigg\Vert_{\textbf{\tf{Q}}_i}^2 -2\bigg(\begin{bmatrix}x^{ex}-\hat{x}^{ex}_{|\mc{F}^y_i}\\\hat{x}^{ex}_{|\mc{F}^y_0}-(\hat{x}^{ex}_{|\mc{F}^y_0})_{|\mc{F}^y_i}\end{bmatrix}\bigg)^T{\Gamma}_i \right\}dt
, \quad 1\leq i \leq N. \label{app:costMinorPartialObs}
\end{multline}
Following the results of Section \ref{sec:CO}, given that $\tbf{A}$, $\tbf{B}_i, 1\leq i \leq N$, and $\tbf{M}$ in the linear dynamics \eqref{app:dynMinorPartialObs} are time-invariant, and the cost functional \eqref{app:costMinorPartialObs} is quadratic, infinite-horizon and time-invariant, we consider the following time-invariant linear state feedback controls as an ansatz for the optimal controls $\hat{u}_{i|\mc{F}^y_i}, 1 \leq i \leq N$,
\begin{equation}\label{app:generalMinorCntrl}
\hat{u}_{i|\mc{F}^y_i} =  L_1 \hat{x}_{i|\mathcal{F}^y_i} + L_2 \hat{x}_{0|\mathcal{F}^y_i} + L_3\sum_{j=1}^{N}\hat{x}_{j|\mathcal{F}^y_i} + L_4 (\hat{x}_{0|\mathcal{F}^y_0})_{|\mathcal{F}^y_i} + L_5 \sum_{j=1}^{N}  (\hat{x}_{j|\mathcal{F}^y_0})_{|\mathcal{F}^y_i} + m.
\end{equation}
Then we substitute \eqref{generalMinorCntrl} in the dynamics of each agent, take the average over population and then its limit as $N \rightarrow \infty$, to obtain the ansatz for the mean field equation as (for the details of the derivation see Section III.A of the manuscript.)
%\end{itemize}
\begin{equation}\label{app:MeanFieldEq}
d\bar{x} = \bar{A} \bar{x}dt + \bar{G} x_0 dt + \bar{H} \hat{x}_{0|\mathcal{F}^y_0} dt + \bar{L} \hat{\bar{x}}_{|\mathcal{F}^y_0} dt + \bar{J}\, \bar{\tilde{x}}^{ex}dt + \bar{m}dt.
\end{equation}

We note that the only time-varying coefficient in \eqref{app:dynMajorPartialObs}-\eqref{app:costMajorPartialObs} and \eqref{app:dynMinorPartialObs}-\eqref{app:costMinorPartialObs} is the Kalman filter gain $K_i$ (diffusion coefficient) due to the the time-varying nature of the estimation error covariance matrix $V_i$. However, due to the Certainty Equivalence Principle the diffusion coefficient does not play a role in determining the coefficients of the control action as shown in Section \ref{sec:CO}. Therefore, the coefficients in the mean field equation \eqref{MeanFieldEq} do not depend parametrically on the filter gain, and  hence do not depend on the estimation error covariance and the observation noise covariance. 

\bibliographystyle{IEEEtran}
\bibliography{bib_Dena_08Jul20}
\end{document}